\documentclass[10pt, a4paper]{article}

\usepackage{amssymb, amsmath, textcomp, amsthm}
\usepackage{bbm,dsfont}
\numberwithin{equation}{section}

\title{Sharp decouplings for three dimensional manifolds in $\R^5$}
\author{Ciprian Demeter, Shaoming Guo and Fangye Shi}
\date{}

\def\R{\mathbb{R}}\def\B{\mathbb{P}}
\def\N{\mathbb{N}}
\def\C{\mathbb{C}}\def\nint{\mathop{\diagup\kern-13.0pt\int}}
\def\Z{\mathbb{Z}}

\def\lesim{\lesssim}

\def\beq{\begin{equation}}
\def\endeq{\end{equation}}
\def\bg{\begin{gathered}}
\def\eg{\end{gathered}}

\def\w{w_{B_N}}

\def\83{\frac{8}{3}}
\def\38{\frac{3}{8}}\def\v{{\bf v}}

\def\mc{\mathcal}

\theoremstyle{plain}
\newtheorem{thm}{Theorem}[section]
\newtheorem{prop}[thm]{Proposition}

\newtheorem{lem}[thm]{Lemma}
\newtheorem{cor}[thm]{Corollary}
\newtheorem{defi}[thm]{Definition}
\newtheorem{claim}[thm]{Claim}
\newtheorem{rem}[thm]{Remark}
\newtheorem{question}[thm]{Question}

\newtheorem*{conj*}{Conjecture}

\newtheorem*{openproblem*}{Open Problem}

\begin{document}

\maketitle

\begin{abstract}
We prove a sharp decoupling for a class of three dimensional manifolds in $\R^5$.
\end{abstract}
\maketitle

\let\thefootnote\relax\footnote{
%Date: \date{\today}\\
The first  author is partially supported  by the NSF Grant DMS-1161752. AMS subject classification: 42A45}

\section{Introduction}

For two symmetric matrices $A_1,A_2\in M_{3}(\R)$ consider the quadratic forms
$$Q_i(r,s,t)=[r,s,t]A_i[r,s,t]^{T}$$
and the associated   three dimensional quadratic surface in $\R^5$ given by
\beq\label{1803e1.1}
\mc{S}=\mc{S}_{A_1,A_2}:=\{(r, s, t, Q_1(r,s,t), Q_2(r,s,t)): (r, s, t)\in [0, 1]^3\}.
\endeq

For a measurable subset $R\subset [0, 1]^3$ and  a measurable function $g: R\to \C$, define the extension operator associated with $R$ and $\mc{S}$ by
\beq\label{1903e1.2}
E_{R}^{\mc{S}} g(x)=\int_{R} g(r, s, t) e(r x_1+s x_2+t x_3+ Q_1(r,s,t)x_4+ Q_1(r,s,t) x_5)drdsdt.
\endeq
Here and throughout the rest of this paper, we will write
$$
e(z)=e^{2\pi iz},\; z\in \R.
$$
For a positive weight $w: \R^5\to (0, \infty)$, define the weighted $L^p$ norm
$$
\|f\|_{L^p(w)}=(\int_{\R^5} |f(x)|^p w(x)dx)^{1/p}.
$$
For a ball $B_N$ centered at $c(B)$ with radius $N$, we let $w_B$ denote the weight
$$
w_B(x)=\frac{1}{(1+\frac{|x-c(B)|}{N})^{C}}.
$$
The exponent $C$ is a large but unspecified constant.
\medskip

Given  $N\ge 1$, $p\ge 2$ and $\mc{S}$ as in \eqref{1803e1.1}, let $D_{\mc{S}}(N, p)$ be the smallest constant such that the following so-called $l^pL^p$ {\em decoupling inequality}
\beq\label{1803e1.6}
\|E_{[0, 1]^3}^\mc{S}g\|_{L^p(w_{B_N})} \le D_\mc{S}(N, p) (\sum_{\substack{\Delta\subset [0, 1]^3\\l(\Delta)=N^{-1/2}}} \|E_{\Delta}^\mc{S}g\|_{L^p(w_{B_N})}^p)^{1/p}
\endeq
holds true for each $g: [0,1]^3\to \C$ and each ball $B_N\subset \R^5$ of radius $N$. Here the summation runs over a finitely overlapping cover of $[0, 1]^3$ by squares $\Delta$ of side length $l(\Delta)=N^{-1/2}$.

The estimate
\beq\label{ff2}
D_\mc{S}(N, 2)\sim 1
\endeq
is an easy consequence of $L^2$ orthogonality, while the estimate
\beq\label{ff3}
D_\mc{S}(N, \infty)\sim N^{3/2}
\endeq
follows from the triangle inequality (upper bound) and from testing \eqref{1803e1.6} with $g\equiv 1$ (lower bound).
Also, we will see in Section \ref{s7} that we have the following universal lower bound
\beq
\label{hhfhdhdcdsjlacmcm;ckj}
D_\mc{S}(N, p)\gtrsim \max\{N^{\frac{3}{2}(\frac{1}{2}-\frac{1}{p})},N^{\frac{3}{2}-\frac{5}{p}}\}.
\endeq

Our main result identifies a large class of manifolds for which this universal lower bound is essentially sharp. It may in fact be the case that this is the largest class of quadratic manifolds with this property. The discussion in the Appendix produces strong evidence in this direction.

\begin{thm}
	\label{1803thm1.1}
	Assume that $Q_1$ and $Q_2$ do not have any common real linear factor. Moreover, assume that for each nonzero vector $(u,v,w)\in\R^3$, the determinant
	$$\det \begin{bmatrix}\frac{\partial Q_1}{\partial r} & \frac{\partial Q_1}{\partial s}& \frac{\partial Q_1}{\partial t}\\\\\frac{\partial Q_2}{\partial r} & \frac{\partial Q_2}{\partial s}& \frac{\partial Q_2}{\partial t}\\\\u & v& w \end{bmatrix}$$
	is not the zero polynomial, when regarded as a function of $r,s,t$. Then for each $\epsilon>0$ and each $p\ge 2$, there exists $C_{\epsilon, p}$ such that
	\beq\label{1803e1.7}
	D_\mc{S}(N, p)\le
	\begin{cases}
		\hfill C_{\epsilon, p} N^{\frac{3}{2}(\frac{1}{2}-\frac{1}{p})+\epsilon}, \hfill &  \text{if } 2\le p\le \frac{14}3\\
		\hfill C_{\epsilon, p} N^{\frac{3}{2}-\frac{5}{p}+\epsilon}, \hfill &  \text{if }  p\ge \frac{14}3
	\end{cases}
	\endeq
\end{thm}

The assumption that $Q_1$ and $Q_2$ do not have any common real linear factor is the same as saying that they do not vanish on any hyperplane at the same time. This is a necessary condition of obtaining decoupling inequalities \eqref{1803e1.7}. To see that it is necessary, we assume that $Q_1$ and $Q_2$ vanish on a hyperplane at the same time, say $\{(r, s, t): t=0\}$. In \eqref{1803e1.6}, we let $g$ be a function supported on the $1/N$ neighbourhood of this hyperplane. Let $B_N$ be the ball of radius $N$ centered at the origin. Hence for every $x\in B_N$ and every $(r, s, t)\in \text{supp}(g)$, it holds that 
\beq
|Q_1(r, s, t)x_4|+|Q_2(r, s, t)x_5|\lesim 1.
\endeq
According to the uncertainty principle, the ball $B_N$ is not able to distinguish the surface $\mc{S}=\mc{S}_{A_1, A_2}$ from $\{(r, s, t, 0, 0): (r, s, t)\in [0, 1]^3\}$. However the best decoupling inequality we can expect for the latter surface and the above function $g$ is given by 
\beq
\|E_{[0, 1]^3}^\mc{S}g\|_{L^p(w_{B_N})} \lesim N^{2(\frac{1}{2}-\frac{1}{p})} (\sum_{\substack{\Delta\subset [0, 1]^3\\l(\Delta)=N^{-1/2}}} \|E_{\Delta}^\mc{S}g\|_{L^p(w_{B_N})}^p)^{1/p}
\endeq
for every $p\ge 2$, which follows easily from an $L^2$ orthogonality argument. In the region $2\le p\le 14/3$, this loss $N^{2(\frac{1}{2}-\frac{1}{p})}$ is much more than what we can afford, which is $N^{\frac{3}{2}(\frac{1}{2}-\frac{1}{p})}$. This proves the necessity of the first assumption on $Q_1$ and $Q_2$. \\

The standard consequence of \eqref{1803e1.7} for exponential sums is discussed in the last section. There are other interesting applications to the decoupling theory of curves that will appear elsewhere.
\smallskip
In the next section we will derive the following corollary.

\begin{cor}\label{1803thm1.11}Let $A_1,A_2\in M_{3}(\R)$ be two symmetric matrices, such that there exists an invertible matrix $M\in GL_{3}(\R)$ satisfying
	\beq
	\label{ff4}
	M^{T}A_iM=\begin{bmatrix}\lambda_{i,1}&0&0\\ 0&\lambda_{i,2}&0\\ 0&0&\lambda_{i,3}\end{bmatrix}, \;\;1\le i\le 2.
	\endeq
	Let $\mc{S}$ be the surface defined in \eqref{1803e1.1}.
	\\
	\\
	(a) Assume  that all the two by two minors of the matrix
	\beq
	\label{ff1}
	\begin{bmatrix}\lambda_{1,1}&\lambda_{1,2}&\lambda_{1,3}\\ \lambda_{2,1}&\lambda_{2,2}&\lambda_{2,3}\end{bmatrix}
	\endeq
	have nonzero determinant.
	Then \eqref{1803e1.7} holds.
	\\
	\\
	(b) If at least one of the two by two minors of \eqref{ff1} is singular then we have
	\beq\label{1803e1.7add}
	\lim_{N\to\infty}\frac{D_\mc{S}(N, p)}{N^{\frac{3}{2}(\frac{1}{2}-\frac{1}{p})}}=\infty
	\endeq
	for each $p>4$.
\end{cor}

The requirement from \eqref{ff4} is rather mild, in particular it does not force $A_1$ and $A_2$ to commute. We refer to \cite{Bec} for a detailed discussion. However inequality \eqref{1803e1.7} also holds true in some cases when $A_1,A_2$ do not satisfy \eqref{ff4}. One such example is
the manifold
$$\{(r, s, t, r^2+s^2, st): (r, s, t)\in [0, 1]^3\},$$
which certainly falls under the scope of Theorem \ref{1803thm1.1}.

Due to \eqref{hhfhdhdcdsjlacmcm;ckj}, the upper bounds in \eqref{1803e1.7} are sharp (apart from the $N^\epsilon$ term). It will suffice to prove the estimate \eqref{1803e1.7} at the critical exponent $\frac{14}3$, as then we can interpolate it with \eqref{ff2} and \eqref{ff3}. We refer the reader to \cite{BD1} for details on how to interpolate decoupling inequalities.
\medskip

The reason we only consider quadratic manifolds is that in some sense they tell the whole story. Indeed, on the one hand \eqref{hhfhdhdcdsjlacmcm;ckj} shows that the decoupling constants for three dimensional manifolds in $\R^5$ do not get smaller in the presence of cubic or higher order terms. In other words, the critical exponent is never larger than $\frac{14}3$. On the other hand, each manifold can be locally approximated  by quadratic manifolds via Taylor's formula, and the general theory can be understood by invoking induction on scales as in \cite{PrSe} (see also Section 7 from \cite{BD1}).

\medskip
Part (b) of Corollary \ref{1803thm1.11} says that if the critical exponent is smaller than $\frac{14}3$, then it is in fact at most 4. These manifolds exhibit various levels of degeneracy, and classifying them will not be our concern here. A more detailed discussion is included in the next section. One surprising example that falls into this category is the very symmetric manifold
$$\{(r, s, t, r^2+s^2+t^2, rs+rt+st): (r, s, t)\in [0, 1]^3\}.
$$
In this case $A_1$ is the identity matrix, so \eqref{ff4} is easily satisfied. The second matrix will have two equal eigenvalues.

One difficulty when approaching three dimensional manifolds in $\R^5$, and in general the $d$-dimensional manifolds in $\R^n$ with $d\not=1,n-1$, is the lack of an appropriate notion of ``curvature". In the case of hypersurfaces ($d=n-1$) decouplings are guided by the principal curvatures, while for curves ($d=1$), by torsion. Similar difficulties have been previously encountered when trying to establish the restriction theory for  manifolds with $1<d<n-1$. We hope that our current work will reignite the interest in this circle of problems.

\medskip

This paper follows the methodology developed by the first author with Jean Bourgain in recent related papers. Most of the material in sections \ref{s2}, \ref{s4} and  \ref{s5} is rather standard. The main new subtleties appear  in Section \ref{s3}. More precisely, Section \ref{s3} addresses the lower dimensional contribution from the Bourgain--Guth-type iteration, where a new difficulty arises. We give a brief description here. As is typical in the multilinear approach, the lower dimensional contribution on a (spatial) ball $B_K$ is coming from (frequency) $K$-cubes  clustered near (that is, lying on the $O(K^{-1})$-neighborhood of) lower dimensional manifolds (in our case these are 2-varieties), as quantified in Theorem \ref{1803thm1.7}. For all practical purposes we may in fact think of these varieties as being planes, as explained in Section \ref{snew}. The main issue is how to estimate such a contribution coming from the $K$-cubes lying in the  $O(K^{-1})$-neighborhood of a fixed plane. There are two major options to start with. The first one is to decouple (separate) the contribution of each of the $K$-cubes. Since we integrate on balls $B_K$, the only decoupling we can perform is the very costly ``trivial decoupling" (see Lemma \ref{1903lemma1.12}). This type of decoupling is simply a manifestation of $L^2$ orthogonality and does not exploit curvature. It turns out that it is not strong enough for our purposes. The other option, and this is the one that we follow, is to perform a Bourgain--Demeter-type  decoupling. This seeks to exploit curvature, but only decouples into frequency cubes having the larger size $K^{-1/2}$. We are thus forced to consider the contribution coming from the $O(K^{-1/2})$-neighborhood of the plane. This scenario also appeared in a simpler context in \cite{BDGu} (see Claim 5.10 there), where the particular nature of the manifold allowed us to estimate the corresponding contribution by invoking dimension-reduction arguments. There is a subtle difference in this context that renders that type of argument useless. To address the issue, we prove that the $O(K^{-1/2})$-wide strip on our manifold is within $O(K^{-1})$ from a certain non degenerate cylinder. The scale $O(K^{-1})$ is now small enough to be accommodated by the uncertainty principle,  when combined with a cylindrical decoupling. The overall argument detailed in Section \ref{s3} is rather delicate, and relies on a careful combination of trivial and Bourgain--Demeter-type decouplings.

In Section \ref{s6} we use linear algebra to prove that the only obstructions to transversality are the 2-varieties. With some extra work we could probably reduce the list of enemies to planes and curves, but we do no pursue this approach. Instead, it turns out that we can control the lower dimensional contribution clustered near each 2-variety, once we can do it for planes. This follows via an approximation argument very similar to the one from \cite{Oh}, that we describe in Section \ref{snew}.

In Section \ref{s7} we describe some related examples and post some open  questions. The Appendix presents strong evidence that the class of manifolds we investigate in this paper contains all manifolds with critical index $\frac{14}{3}$.

\bigskip

\section{Linear algebra reductions}
In this section we demonstrate that the decoupling theory is essentially invariant under certain transformations. This will allow us to give a simple proof of Corollary \ref{1803thm1.11} using Theorem \ref{1803thm1.1}.

\begin{prop}
	\label{ff8}
	Let $A_1,A_2\in M_3(\R)$, $M\in GL_3(\R)$ and $\beta=[\beta_{ij}]_{1\le i,j\le 2}\in GL_2(\R)$. Define
	$$
	B_i:=M^{T}(\beta_{i,1}A_1+\beta_{i,2}A_2)M,\;\;1\le i\le 2.$$
	Let $D_{1}(N,p)$ and $D_2(N,p)$ be the decoupling constants associated with $\mc{S}_{A_1,A_2}$ and $\mc{S}_{B_1,B_2}$, respectively. Then for each $p\ge 2$
	$$D_{1}(N,p)\sim_{p,M,\beta} D_2(N',p),$$
	where $N'\sim_{M,\beta} N$.
\end{prop}
\begin{proof}
	Denote by $E^{(1)}$ and $E^{(2)}$ the extension operators associated with the two surfaces.
	For each square $R\subset [0,1]^3$ we may write, denoting $v=(r,s,t)$ and using the changes of variables
	$$v=L_M(w):=wM^{T}$$
	and
	$$[x_4,x_5]=[y_4,y_5]\beta,\;\;\;\;[y_1,y_2,y_3]=[x_1,x_2,x_3]M$$
	$$E^{(1)}_{R}g(x_1,\ldots,x_5)=\int_{R}g(v)e(v\cdot(x_1,x_2,x_3)+vA_1v^{T}x_4+vA_2v^{T}x_5)dv=$$
	$$\det(M)\int_{(L_M)^{-1}R}g\circ L_M(w)e(w\cdot(y_1,y_2,y_3)+wM^{T}A_1Mw^{T}x_4+wM^{T}A_2Mw^{T}x_5)dw$$
	$$=\det(M)\int_{(L_M)^{-1}R}g\circ L_M(w)e(w\cdot(y_1,y_2,y_3)+wB_1w^{T}y_4+wB_2w^{T}y_5)dw$$
	$$=det(M)E^{(2)}_{(L_M)^{-1}R}g\circ L_M(y_1,y_2,y_3,y_4,y_5).$$
	The proposition will now follow once we make two observations. First, since $\beta$ and $M$ are nonsingular, the transformation
	$$T(x_1,\ldots,x_5)=(y_1,y_2,y_3,y_4,y_5)$$
	has finite distortion. In particular, for each ball $B\subset \R^5$
	$$w_B(T^{-1}y)\sim w_B(y).$$
	Second, $(L_M)^{-1}R$ will be a parallelogram with area comparable to the area of $R$, and which sits inside a square $R'$ with side length comparable to that of $R$. In particular, if $l(R)=N^{-1/2}$ then
	$$\|E^{(2)}_{(L_M)^{-1}R}h\|_{L^p(w_{B_N})}\lesssim \|E^{(2)}_{R'}h\|_{L^p(w_{B_N})}.$$
	This can be seen by observing that $F_1=E^{(2)}_{(L_M)^{-1}R}h$ and $F_2=E^{(2)}_{R'}h$ are related via
	$$\widehat{F_1}=\widehat{F_2}1_P,$$
	with $P$ a rectangular box in $\R^5$ having three side lengths comparable to $N^{-1/2}$ and two of them comparable to $N^{-1}$.
	
	The details are left to the interested reader.

\end{proof}

As a first application of this result, we prove part (b) of Corollary \ref{1803thm1.11}. It is rather immediate that the existence of a singular two by two minor of \eqref{ff1} leads to the existence of a $\beta\in GL_2(\R)$ so that the matrix
$$\beta\begin{bmatrix}\lambda_{1,1}&\lambda_{1,2}&\lambda_{1,3}\\ \lambda_{2,1}&\lambda_{2,2}&\lambda_{2,3}\end{bmatrix}$$
is of one of the following types
\beq
\begin{split}
	& \begin{bmatrix}0&0&0\\ a&b&c\end{bmatrix}, \begin{bmatrix}0&0&c\\ a&b&0\end{bmatrix}, \begin{bmatrix}0&c&0\\ a&0&b\end{bmatrix}, \begin{bmatrix}c&0&0\\ 0&a&b\end{bmatrix},
\end{split}
\endeq
with $a,b,c\in\{0,1,-1\}$.
In the first case, the decoupling constant of $\mc{S}$ will be comparable to that of the manifold in $\R^4$
$$\{(r,s,t,ar^2+bs^2+ct^2),\;0\le r,s,t\le 1\}.$$
The most favorable case is when $a,b,c\not=0$, when most curvature is present. In \cite{BD10} it is proved that the critical index for this manifold is $\frac{10}3$. In particular,
$$\lim_{N\to\infty}\frac{D_{\mc{S}}(N,p)}{N^{\frac32(\frac12-\frac1p)}}=\infty,\;\;p>\frac{10}{3}.$$

The remaining three cases are symmetric, so it suffices to consider the first one. The decoupling constant of $\mc{S}$ will in this case be comparable to that of the product-type manifold
$$\mc{S}_{prod}:=\{(r,s,t,ar^2+bs^2, ct^2),\;0\le r,s,t\le 1\}.$$
Let
$$\mc{S}_1=\{(r,s,ar^2+bs^2),\;0\le r,s\le 1\}$$
and
$$\mc{S}_2=\{(t,ct^2),\;0\le t\le 1\}.$$
By testing \eqref{1803e1.6} with functions of the form $g(r,s,t)=g_1(r,s)g_2(t)$ we see that
$$D_{\mc{S}_{prod}}(N,p)\gtrsim D_{\mc{S}_1}(N,p)D_{\mc{S}_2}(N,p).$$
The values of $D_{\mc{S}_1}(N,p)$ and $D_{\mc{S}_2}(N,p)$ are smallest when $a,b,c\not=0$, which guarantees most curvature. But even in this case, the results in \cite{BD10} show that
$$\lim_{N\to\infty}\frac{D_{\mc{S}_1}(N,p)}{N^{\frac12-\frac1p}}=\infty,\;\;p>4$$
$$D_{\mc{S}_2}(N,p)\gtrsim N^{\frac12(\frac12-\frac1p)}, \;\;p\ge 2.$$
Combining these leads to the desired estimate
$$\lim_{N\to\infty}\frac{D_{\mc{S}_{prod}}(N,p)}{N^{\frac{3}{2}(\frac12-\frac1p)}}=\infty,\;\;p>4$$
%Reasoning like in Section \ref{s7},
%$$\|\sum_{n_1=1}^{N}e(n_1)\sum_{n_2=1}^{N}\sum_{n_3=1}^{N}\|_{L^p([0,1]^3)}$$
\bigskip

Proposition \ref{ff8} also has the following rather immediate consequence.
\begin{cor}
	Let $A_1,A_2\in M_3(\R)$ satisfy the requirement of part (a) of Corollary \ref{1803thm1.11}, and let $\mc{S}_{A_1,A_2}$ be the associated surface. Then
	$$D_{\mc{S}_{A_1,A_2}}(N,p)\sim_{p,M,\beta} D_{\mc{S}}(N',p),$$
	where $N'\sim_{A_1,A_2} N$
	and
	\beq
	\label{ff9}
	\mc{S}:=\{(r, s, t, \frac{1}2 (r^2+A s^2),\frac{1}2 ( t^2+ B s^2)): (r, s, t)\in [0, 1]^3\},
	\endeq
	for some $A,B\not=0$ depending on $A_1,A_2$.
\end{cor}

It is now immediate that part (a) of Corollary \ref{1803thm1.11} will follow from Theorem \ref{1803thm1.1}.
\bigskip
\begin{rem}
	\label{rmnf897-594090-ei=rfo=[=1[}
	It is easy to see that the requirement in Theorem \ref{1803thm1.1} is invariant under nonsingular linear changes of variables. Indeed, assume $Q_1,Q_2$ satisfy this requirement, and let $B\in M_3(\R)$ be nonsingular. Define $\widetilde{Q_i}(r,s,t)=Q_i(B[r,s,t]^{T})$. It now suffices to note that
	$$\begin{bmatrix}\frac{\partial \widetilde{Q_1}}{\partial r}(\v) & \frac{\partial \widetilde{Q_1}}{\partial s}(\v)& \frac{\partial \widetilde{Q_1}}{\partial t}(\v)\\\\\frac{\partial \widetilde{Q_2}}{\partial r}(\v) & \frac{\partial \widetilde{Q_2}}{\partial s}(\v)& \frac{\partial \widetilde{Q_2}}{\partial t}(\v)\\\\u & v& w \end{bmatrix}=\begin{bmatrix}\frac{\partial Q_1}{\partial r}(\v') & \frac{\partial Q_1}{\partial s}(\v')& \frac{\partial Q_1}{\partial t}(\v')\\\\\frac{\partial Q_2}{\partial r}(\v') & \frac{\partial Q_2}{\partial s}(\v')& \frac{\partial Q_2}{\partial t}(\v')\\\\u' & v'& w' \end{bmatrix}B$$
	where $[u,v,w]=[u',v',w']B$ and $\v=(r,s,t)$, $\v'=B[r,s,t]^{T}$.
\end{rem}

\bigskip
The rest of the paper will be concerned with the proof  of Theorem \ref{1803thm1.1}.

\section{Transversality}
\label{s2}
Let $m$ be a positive integer. For $1\le j\le m$, let $V_j$ be a $d$-dimensional linear subspace of $\R^n$. Let also $\pi_j: \R^n\to V_j$ denote the orthogonal projection onto $V_j$. Define
\beq
\Lambda(f_1, f_2, ..., f_m)=\int_{\R^n} \prod_{j=1}^m f_j (\pi_j (x))dx,
\endeq
for $f_j: V_j\to \C$. We recall the following theorem from Bennett, Carbery, Christ and Tao \cite{BCCT}.
\begin{thm}
	\label{nndkjlaleqffhqeuvasdk}
	Given $p\ge 1$, the estimate
	\beq\label{1803e1.8}
	|\Lambda(f_1, f_2, ..., f_m)| \lesim \prod_{j=1}^m \|f_j\|_{p},
	\endeq
	holds if and only if $np=dm$ and the following Brascamp-Lieb transversality condition is satisfied
	\beq\label{1803e1.9}
	dim(V) \le \frac{1}{p} \sum_{j=1}^m dim(\pi_j(V)), \text{ for each linear subspace } V\subset \R^n.
	\endeq
\end{thm}
An equivalent formulation of the estimate \eqref{1803e1.8} is
\beq\label{1712e1.9}
\|(\prod_{j=1}^m g_j\circ \pi_j)^{1/m} \|_q \lesim (\prod_{j=1}^m \|g_j\|_2)^{1/m},
\endeq
with $q=\frac{2n}{d}.$
The restriction that $p\ge 1$ becomes $dm\ge n$. The transversality condition \eqref{1803e1.9} becomes
\beq\label{1712e1.11}
dim(V)\le \frac{n}{d m}\sum_{j=1}^m dim(\pi_j(V)), \text{ for each subspace } V\subset \R^n.
\endeq
Now let us be more specific about $d, n$ and $m$.  In this section, we will take $n=5$, since we are considering a three dimensional surface $\mc{S}$  in $\R^5$. We will take $d=3$, since the tangent space to $\mc{S}$ has dimension three. The degree $m$ of multi-linearity is more complicated. It will not be a fixed integer, but will rather depend on the scale of the sets (cubes) we are using.
\medskip

With this numerology \eqref{1712e1.9} becomes
\beq\label{1712e1.12}
\|(\prod_{j=1}^m g_j\circ \pi_j)^{1/m} \|_{L^{\frac{10}3}(\R^5)} \lesim (\prod_{j=1}^m \|g_j\|_2)^{1/m},
\endeq
and the condition \eqref{1712e1.11} becomes
\beq\label{1712e1.13}
dim(V)\le \frac{5}{3m}\sum_{j=1}^m dim(\pi_j(V)), \text{ for each subspace } V\subset \R^5.
\endeq
\bigskip

Fix now $\mc{S}$ satisfying the requirement of Theorem \ref{1803thm1.1}.
We will next try to understand what it mean for \eqref{1712e1.13} to be satisfied, given that $V_j$ are the tangent spaces to $\mc{S}$ at the points $(r_j,s_j,t_j)\in[0,1]^3$. We will see that this means that a rather big fraction of these points  should not belong to a  2-variety. By that we will mean the (real) zero set of a nontrivial  polynomial $P(r,s,t)$ of degree at most two.

In order to achieve this, we need more notation.
Let $\mc{M}$ be an $m\times n$ matrix with $m\ge n$. We define $\det(\mc{M})$ to be the $l^1$ sum of the determinants of all  $n\times n$ sub-matrices of $\mc{M}$.

At one point $(r, s, t)\in [0, 1]^3$, we denote by $n_1, n_2$ and $n_3$ the three tangent vectors of the surface $\mc{S}$ given by
\beq
\begin{split}
	& n_1=(1, 0, 0, \frac{\partial{Q_1}}{\partial r}, \frac{\partial{Q_2}}{\partial r}),\\
	& n_2=(0, 1, 0, \frac{\partial{Q_1}}{\partial s}, \frac{\partial{Q_2}}{\partial s}),\\
	& n_3=(0, 0, 1, \frac{\partial{Q_1}}{\partial t}, \frac{\partial{Q_2}}{\partial t}).
\end{split}
\endeq
The tangent space they span will be denoted by $V_{r,s,t}$. The projection onto this space will be denoted by $\pi_{r,s,t}$.
\medskip

\noindent For a one dimensional subspace $V\subset \R^5$ spanned by a unit vector $x$, denote by $\mc{M}_V(r,s,t)$ the $1\times 3$ matrix
\beq
[x\cdot n_1, x\cdot n_2, x\cdot n_3].
\endeq
For a two dimensional subspace $V\subset \R^5$ spanned by two orthogonal unit vectors $x, y\in \R^5$, denote by $\mc{M}_V(r,s,t)$ the $2\times 3$ matrix
\beq\label{1803e1.15}
\begin{bmatrix}
	x\cdot n_1 & x\cdot n_2 & x\cdot n_3\\
	y\cdot n_1 & y\cdot n_2 & y\cdot n_3
\end{bmatrix}
\endeq
Similarly, for a four dimensional subspace $V\subset \R^5$ spanned by four orthogonal unit vectors $x, y, z, \theta \in \R^5$, we denote by $\mc{M}_V(r,s,t)$ the $4\times 3$ matrix
\beq\label{1803e1.16}
\begin{bmatrix}
	x\cdot n_1 & x\cdot n_2 & x\cdot n_3\\
	y\cdot n_1 & y\cdot n_2 & y\cdot n_3\\
	z\cdot n_1 & z\cdot n_2 & z\cdot n_3\\
	\theta\cdot n_1 & \theta\cdot n_2 & \theta\cdot n_3
\end{bmatrix}
\endeq

\begin{rem}
	\label{ww14}
	Note that for $V\subset \R^5$ of dimensions $1,2$ or $4$, the condition $\det(\mc{M}_V(r,s,t))\not=0$ is equivalent with $dim(\pi_{r,s,t}(V))$ being at least $1,2$ or $3$, respectively. This is a consequence of the rank-nullity theorem.
\end{rem}

Now we are ready to state our transversality condition.
\begin{defi}[Transversality]\label{1803defi1.3}
	A collection of $m\ge 10^4 $  sets $S_1, ..., S_m\subset [0, 1]^3$ is said to be $\nu$-transverse if for each
	\beq
	1\le i_1\neq i_2\neq ... \neq i_{[m/100]}\le m,
	\endeq
	we have that
	for each subspace $V\subset \R^5$ of dimension one, two or four,
	\beq\label{2201e4.2}
	\max_{1\le j\le [m/100]} \inf_{(r, s, t)\in S_{i_j}} \left|\det(\mc{M}_V(r,s,t)) \right| \ge \nu.
	\endeq
\end{defi}

We next observe that the transversality condition in Definition \ref{1803defi1.3} is stronger than the Brascamp-Lieb transversality condition \eqref{1712e1.13}.

\begin{prop}
	Consider $m$ sets $S_j$ which are $\nu$-transverse for some $\nu>0$. Then for each $(r_j,s_j,t_j)\in S_j$, the $m$ tangent planes $V_j, 1\le j\le m$ spanned by the vectors $n_i(r_j, s_j, t_j),\;1\le i\le 3,$ satisfy the condition \eqref{1712e1.13}.
\end{prop}
\begin{proof}
	The case $dim(V)=5$ is trivially true, as we always have $dim(\pi_j(V))=3$ for all $j$. When $dim(V)=4$, in order to verify \eqref{1712e1.13}, it suffices to prove that there are at least $9m/10$  $V_j$ with $dim(\pi_j(V))\ge 3$. This follows from Remark \ref{ww14} and \eqref{2201e4.2}. The cases $dim(V)=1, 2, 3$ can be proved similarly.
\end{proof}
An $\alpha$-cube is defined to be a closed cube with side length $\frac1{\alpha}$ inside $[0, 1]^3$. If $\alpha\in 2^{\Z}, $ the collection of all dyadic $\alpha$-cubes will be denoted by $Col_\alpha$. We will implicitly assume that various values of $\alpha$ we use are in $2^{\Z}$.
\medskip

The following result provides a nice criterium for transversality.

\begin{thm}\label{1803thm1.7}
	Consider an arbitrary collection ${\mathcal C}$ of $m(\ge 10^4)$  $K$-cubes such  that the $10/K$ neighbourhood of each 2-variety in $\R^3$ intersects no more than $m/100$ of these $K$-cubes.  Then the cubes in ${\mathcal C}$ are $\nu_K$-transverse, for some $\nu_K>0$ that depends only on $K$.
\end{thm}
\begin{proof}
	The proof will follow from a standard compactness argument combined with Lemma \ref{2201lemma3.1}.
	
\end{proof}
\medskip

For each subset $R\subset [0, 1]^3$ and $0<\delta<1$, let $\mc{N}_{R, \delta}$ be a $\delta$-neighbourhood of
$$
\mc{S}_R=\{(r, s, t, Q_1(r,s,t),Q_2(r,s,t)): (r, s, t)\in R\}.
$$

The following multilinear restriction theorem is a particular case of a result from \cite{BBFL}. Its proof relies on Theorem \ref{nndkjlaleqffhqeuvasdk} and induction on scales.
\begin{thm}\label{2001thm1.9}
	Let $R_j$ with $j=1, ..., m$ be a collection of subsets of $[0,1]^3$ that are $\nu$-transverse. For each $f_j: \mc{N}_{R_j, 1/N}\to \C$, each $\epsilon>0$ and each ball $B_N\subset \R^5$ of radius $N\ge 1$, we have
	\beq\label{1803e1.26}
	\|\prod_{j=1}^m |\hat{f}_j|^{1/m}\|_{L^{\frac{10}3}(B_N)}\lesim_{\epsilon,\nu} N^{-1+\epsilon}\left( \prod_{j=1}^m\|f_j\|_{L^2(
		\mathcal{N}_{R_j, 1/N})} \right)^{1/m}.
	\endeq
\end{thm}

We close this section with presenting the following consequence, a direct application of  Proposition 6.5 from \cite{BD2} with $n=5$, $d=3$ and
$$\kappa_p=\frac{p-\frac{10}{3}}{p-2},\;\;p\ge \frac{10}{3}.$$

This result will play a key role in the iteration from Section \ref{s5}.
\begin{prop}
	\label{iovjurgyptn8vbgu89357893v7589ty7056893}
	Let $R_j$ with $j=1, ..., m$ be a collection of subsets of $[0,1]^3$ that are $\nu$-transverse.
	For  each ball $B_R$ in $\R^5$ with radius $R\ge N\ge 1$, $p\ge \frac{10}{3}$, $\epsilon>0$, $\kappa_p\le \kappa\le 1$ and $g_i:R_i\to \C$ we have
	$$\|(\prod_{i=1}^{m}\sum_{\atop{l(\tau)=N^{-1/2}}}|E_{\tau}g_i|^2)^{\frac1{2m}}\|_{L^{p}(w_{B_R})}\lesssim_{\epsilon,\nu}$$
	$$
	N^{\epsilon}\|(\prod_{i=1}^{m}\sum_{\atop{l(\Delta)=N^{-1}}}|E_{\Delta}g_i|^2)^{\frac1{2m}}
	\|_{L^{p}(w_{B_R})}^{1-\kappa}
	(\prod_{i=1}^{m}\sum_{\atop{l(\tau)=N^{-1/2}}}\|E_{\tau}g_i\|_{L^{p}(w_{B_R})}^2)^{\frac{\kappa}{2m}}.
	$$
\end{prop}

\section{Lower dimensional decoupling}
\label{s3}
Recall that we are working with a manifold $$\mc{S}:=\{(r, s, t, Q_1(r,s,t),Q_2(r,s,t)): (r, s, t)\in [0, 1]^3\}$$
satisfying the requirement in Theorem \ref{1803thm1.1}.
The assumption that $Q_1$ and $Q_2$ do not vanish on any hyperplane at the same time, implies that there exists $\eta>0$ such that for any $\alpha,\beta,\gamma = O(1)$, either $Q_1 (r,s,\alpha+\beta r+\gamma s)$ or $Q_2 (r,s,\alpha+\beta r+\gamma s)$, when viewed as polynomials in $r,s$, will have at least one quadratic coefficient which has absolute value at least $\eta$.\
\smallskip

Unless specified otherwise, the extension operator $E$ will refer to $E^{\mc{S}}$.\
\smallskip

The main result of this section is the following   decoupling inequality for cubes clustered near a plane. It will be used in the next section in the proof of Proposition \ref{1903prop1.10}.
\begin{thm}\label{1903lemma1.14}
	Let $H$ be a  plane in $\R^3$ which intersects $[0, 1]^3$. Fix a large constant $K\gg 1$. Let $\mc{R}\subset Col_{K^{1/2}}$ be a collections of $K^{1/2}$-cubes, each of which intersects $H$. Then we have the decoupling inequality
	\beq\label{1903e1.39}
	\|\sum_{R\in \mc{R}}E_{R}g\|_{L^p(w_{B_K})}\lesim_\epsilon K^{\frac{3}{2}(\frac{1}{2}-\frac{1}{p})+\epsilon} (\sum_{R\in \mc{R}} \|E_{R}g\|_{L^p(w_{B_K})}^p)^{1/p},
	\endeq
	for all $4\le p\le 6$.
\end{thm}
We will only use this result for $p=\frac{14}{3}$. Let us comment on the strength of this result. It is stronger than what we would get by using only trivial decoupling, and by that we refer to Lemma \ref{1903lemma1.12} below. Indeed this lemma gives the poor decoupling  constant $K^{2(\frac{1}{2}-\frac{1}{p})}$,  because it exploits no curvature. % % In our case the plane $H$ reduces the number of variables $r,s,t$ to two. % % We will end up performing trivial decoupling in one direction, and exploiting curvature in the other direction.
\medskip

Given a manifold $$\mc{M}:=\{(v,Q(v)):v\in\R^{d}\}$$
associated with $Q:\R^{d}\to \R^{d'}$, its extension operator will be defined as follows
$$E_V^{\mc{M}}g(x,x')=\int_Vg(v)e(xv+x'Q(v))dv.$$
Here $V$ is an arbitrary measurable set in $\R^d$, $g$ is an arbitrary complex valued measurable function on $\R^d$ and $(x,x')\in\R^{d}\times \R^{d'}.$
We recall the following dimension reduction result, which is a small variation of the one from \cite{BDGu}.
\begin{lem}
	\label{abc8}
	Let $p\ge q\ge 2.$
	Let $Q_i:\R^3\to\R$, $i=1,2$ be measurable. Fix  $U_1,\ldots,U_l$, an arbitrary measurable partition of $[0,1]^{3}$ and fix $B$,  an arbitrary measurable subset of $\R^{4}$. For $i=1,2$, let $E^{(i)}=E^{\mc{M}_i}$ denote the extension operators associated with  the manifolds $\mc{M}_i$ defined as follows
	$$\mc{M}_1=\{(u,Q_1(u)):u\in\R^{3}\},$$
	$$\mc{M}_2=\{(u,Q_1(u),Q_2(u)):u\in\R^{3}\}.$$
	Fix a measurable function $h:[0,1]^3\to\C$.
	Let $C$ be a number such that the inequality
	$$\|E^{(1)}_{[0,1]^{3}}\widetilde{h}\|_{L^p(B)}\le C(\sum_i\|E_{U_i}^{(1)}\widetilde{h}\|_{L^p(B)}^q)^{1/q}$$
	holds for all measurable $\widetilde{h}$ such that $|\widetilde{h}|=|h|$.
	
	Then for each measurable set $B'\subset \R$  we have
	$$\|E^{(2)}_{[0,1]^{3}}h\|_{L^p(B\times B')}\le C(\sum_i\|E_{U_i}^{(2)}h\|_{L^p(B\times B')}^q)^{1/q}.$$
\end{lem}

We will also need the following instances of cylindrical decouplings.
\begin{lem}
	\label{ww8}
	Consider the curve $\gamma$ in the $(u_1,u_2)$-plane
	$$\gamma:=\{(u_1,\psi(u_1)):\;|u_1|\lesssim 1\}.$$
	We assume $|\psi''|\sim 1.$ For $K\gg 1$, let $I_1,I_2,\ldots$ be a partition of $|u_1|\lesssim 1$ using intervals of length $\sim K^{-1/2}$. Partition the $O(K^{-1})$-neighborhood of $\gamma$ into sets $R_i$, each of which is an $O(K^{-1})$ neighborhood of $I_i$. Note that each $R_i$ looks like a $\sim K^{-1/2}\times K^{-1}$ rectangle. For each $R_i$ consider the vertical region $P_i$ in $\R^4$ defined as follows
	$$P_i=\{(u_1,u_2,u_3,u_4):\;(u_1,u_2)\in R_i,\;u_3,u_4\in\R\}.$$
	For each $f:\R^4\to\C$ with Fourier transform supported in $\cup_i P_i$, we will define the Fourier restriction $f_{P_i}$ of $f$ to $P_i$ by
	$$\widehat{f_{P_i}}=\widehat{f}1_{P_i}.$$
	Then for each $2\le p\le 6$, each such $f$ and each $B_K$ in $\R^4$ we have
	$$\|f\|_{L^p(w_{B_K})}\lesssim_\epsilon K^{\epsilon}(\sum_{i}\|f_{P_i}\|_{L^p(w_{B_K})}^2)^{1/2}$$
	and
	\beq
	\label{ww13}
	\|f\|_{L^p(w_{B_K})}\lesssim_\epsilon K^{\frac14-\frac1{2p}+\epsilon}(\sum_{i}\|f_{P_i}\|_{L^p(w_{B_K})}^p)^{1/p}.
	\endeq
	
\end{lem}
\begin{proof}
	The first inequality follows immediately by applying Theorem 1.1 from \cite{BD1} (in the form from Section 7)  combined with a standard Fubini-type argument. The second one follows from the first one combined with H\"older.
	
\end{proof}

\begin{lem}
	\label{cyl}
	Consider the surface $\Lambda$ where
	$$\Lambda:=\{(u_1,u_2,\psi(u_1,u_2)):\;|u_1|,|u_2|\lesssim 1\}.$$
	We assume $|D^2(\psi)|\sim 1$ where $D^2(\psi)$ is the Hessian of $\psi$. For $K\gg 1$, let $I_1,I_2,\ldots$ be a partition of $|u_1|,|u_2|\lesssim 1$ using squares of side length $\sim K^{-1/2}$. Partition the $O(K^{-1})$-neighborhood of $\Lambda$ into sets $R_i$, each of which is an $O(K^{-1})$ neighborhood of $I_i$. Note that each $R_i$ looks like a $\sim K^{-1/2}\times K^{-1/2} \times K^{-1}$ rectangular box. For each $R_i$ consider the vertical region $P_i$ in $\R^4$ defined as follows
	$$P_i=\{(u_1,u_2,u_3,u_4):\;(u_1,u_2,u_3)\in R_i,\;u_4\in\R\}.$$
	For each $f:\R^4\to\C$ with Fourier transform supported in $\cup_i P_i$, we will define the Fourier restriction $f_{P_i}$ of $f$ to $P_i$ by
	$$\widehat{f_{P_i}}=\widehat{f}1_{P_i}.$$
	Then for each $p\geq 4$, each such $f$ and each $B_K$ in $\R^4$ we have
	\beq
	\label{ww14bju9fgjb89tibv90rt890u}
	\|f\|_{L^p(w_{B_K})}\lesssim_\epsilon K^{1-\frac3{p}+\epsilon}(\sum_{i}\|f_{P_i}\|_{L^p(w_{B_K})}^p)^{1/p}.
	\endeq
	
\end{lem}
\begin{proof}
	The inequality follows immediately by applying Theorem 1.1 from \cite{BD10} combined with a standard Fubini-type argument.
	
\end{proof}

\begin{rem}
	It may help to realize that the Fourier transform of the function $f$ from the lemmas is supported in the $O(K^{-1})$-neighborhood of the cylinder
	$$Cyl=\{(u_1,\psi(u_1),u_3,u_4),\;|u_1|\lesssim 1,\;u_3,u_4\in\R\}.$$
	and
	$$Cyl=\{(u_1,u_2,\psi(u_1,u_2),u_4),\;|u_1|,|u_2|\lesssim 1,\;u_4\in\R\},$$
	respectively.
	These cylinders are obtained in the first case by attaching to each point $(u_1,\psi(u_1),0,0)$, the plane  $\pi$   spanned by $(0,0,1,0)$ and $(0,0,0,1)$, while in the second case by attaching to each point $(u_1,u_2,\psi(u_1,u_2),0)$ the line $L$ spanned by $(0,0,0,1)$. We will call this plane (line) the ``vertical component" of $Cyl$.
	
	The results of the lemmas remain true if the  cylinder is replaced with any of its rigid motions.
\end{rem}

\bigskip

We will now start the proof of Theorem \ref{1903lemma1.14}.
\medskip

By symmetry, we could assume our plane is given by $t=\alpha+\beta r+\gamma s$ for some $\alpha,\beta,\gamma= O(1)$. And without loss of generality, we could also assume that $$\widetilde{Q}_1 (r,s):= Q_1(r,s,\alpha+\beta r+\gamma s)=a r^2+b s^2+c rs + L(r,s),$$ with $\max\{|a|,|b|,|c|\}>\eta$ and $L$  affine. The value of $L$ is irrelevant (it never influences the curvature) and can be discarded.  Now we will analyze three cases. Let us start by briefly explaining the third case, which is conceptually the easiest. When the quantity $c^2-4ab$ is away from zero, we can view the relevant manifold (living in $\R^4$) as being close to a cylinder over a two dimensional surface (lying inside a three dimensional space). The requirement on $c^2-4ab$ being nonzero is equivalent to the non degeneracy of the cylinder. We will then combine the well established decoupling theory for surfaces\footnote{What matters in all three cases is that at least one principal curvature of the surface is away from zero. We can afford to perform a trivial decoupling in the direction corresponding to small curvature} with the cylindrical decoupling from Lemma \ref{cyl}. On the other hand, when $c^2-4ab=0$ the cylinder is degenerate, it lives inside a copy of $\R^3$. We will then essentially view it as a cylinder over a curve, and will instead invoke Lemma \ref{ww8}. \\\

Case 1. Assume $|a|> \eta/4$. Suppose $Q_1 (r,s,t)= Ar^2 +B s^2 +C t^2 +D rs + E rt + F st$ for some $A,B,C,D,E,F \in \R$. Then by a direct computation,
\beq
\label{a}
a=A+C\beta ^2 +E\beta .
\endeq

Tile the unit square $\{(r,s)\in[0,1]\times [0,1]\}$ with $K^{1/2}-$squares, and call this collection $\mc{R}_{tile}$. By allowing another $O(1)$ loss, we may in fact assume that there is at most one $R\in \mc{R}$ whose $(r,s)-$projection is any given  square in $\mc{R}_{tile}$.

Let
$$h=g\sum_{R\in\mc{R}}1_R.$$
With this in mind, it suffices to prove that for $2\le p\le 6$ (note that in this case we can afford a more generous range than $4\le p\le 6$)
$$\|E_{[0,1]^3}h\|_{L^p(w_{B_K})}\lesim_\epsilon K^{\frac{3}{2}(\frac12-\frac{1}{p})+\epsilon} (\sum_{I,J\subset [0,1]\atop{|I|=|J|=K^{-1/2}}} \|E_{J\times I\times [0,1]}h\|_{L^p(w_{B_K})}^p)^{1/p}.$$

By invoking Lemma \ref{abc8}, it will suffice to prove the following inequality for each $\widetilde{h}$ with $|\widetilde{h}|=|h|$
\beq
\label{ww6}
\|E_{[0,1]^3}^{(1)}\widetilde{h}\|_{L^p(w_{B_K})}\lesim_\epsilon K^{\frac{3}{2}(\frac12-\frac{1}{p})+\epsilon} (\sum_{I,J\subset [0,1]\atop{|I|=|J|=K^{-1/2}}} \|E^{(1)}_{J\times I\times [0,1]}\widetilde{h}\|_{L^p(w_{B_K})}^p)^{1/p},
\endeq
where $E^{(1)}$ is the extension operator for the manifold
$$\mc{M}_1=\{(r,s,t,Q_1(r,s,t)),\;(r,s,t)\in\R^3\}.$$
Of course, our estimates need to be uniform over $\alpha,\beta,\gamma$.
As a first step towards proving \eqref{ww6}, we perform a trivial decoupling in the $s$ direction (Lemma \ref{1903lemma1.12}), to write for each $p\ge 2$
$$\|E^{(1)}_{[0,1]^3}\widetilde{h}\|_{L^p(w_{B_K})}\lesim K^{\frac{1}{2}-\frac{1}{p}} (\sum_{I\subset [0,1]\atop{|I|=K^{-1/2}}} \|E^{(1)}_{[0,1]\times I\times [0,1]}\widetilde{h}\|_{L^p(w_{B_K})}^p)^{1/p}.$$

Fix $I=[s_0,s_0+K^{-1/2}]$ from the summation. It remains to prove the following inequality, for $2\le p\le 6$
\beq
\label{ww5}
\|E^{(1)}_{[0,1]\times I\times [0,1]}\widetilde{h}\|_{L^p(w_{B_K})}\lesssim_{\epsilon} K^{\frac14-\frac1{2p}+\epsilon}(\sum_{J\subset [0,1]\atop{|J|=K^{-1/2}}}\|E^{(1)}_{J\times I\times [0,1]}\widetilde{h}\|_{L^p(w_{B_K})}^p)^{1/p}.
\endeq
Consider the following strip on $\mc{M}_1$
$$\mc{M}_{1,I}=\{(r,s,t,Q_1(r,s,t)):\;(r,s,t)\in [0,1]\times I\times [0,1],\;|t-(\alpha+\beta r+\gamma s)|\lesssim K^{-1/2}\}.$$
This lies in the $O(K^{-1/2})$-neighborhood of the parabola
$$\B=\{(r,s_0,\alpha+\beta r +\gamma s_0, ar^2+b s_0 ^2+ crs_0):\;r\in[0,1]\},$$
whose curvature satisfies
\beq
\label{ww11}
\kappa\sim 1.
\endeq
because $|a|$ is away from zero by assumption.
This parabola lies in a translate of the plane spanned by
$$w_1=(1,0,\beta,0)$$
and
$$w_2=(0,0,0,1).$$
Tile $\mc{M}_{1,I}$ with caps
$$\mc{M}_{1,I,J}=\{(r,s,t,Q_1(r,s,t)):\;(r,s,t)\in J\times I\times [0,1],\;|t-(\alpha+\beta r+\gamma s)|\lesssim K^{-1/2}\}.$$
Note that the Fourier transform of $E^{(1)}_{J\times I\times [0,1]}\widetilde{h}$ is supported on the cap $\mc{M}_{1,I,J}$. By loosing $O(1)$ we may assume that the $K^{-1}$-neighborhoods $N_{I,J}$ of these caps are pairwise disjoint. Thus \eqref{ww5} will follow if we prove that for each $f$ Fourier supported in $\cup_{J}N_{I,J}$ we have
\beq
\label{ww12}
\|f\|_{L^p(w_{B_K})}\lesssim_{\epsilon} K^{\frac14-\frac1{2p}+\epsilon}(\sum_{J\subset [0,1]\atop{|J|=K^{-1/2}}}\|f_{N_{I,J}}\|_{L^p(w_{B_K})}^p)^{1/p},
\endeq
where $f_{N_{I,J}}$ is the Fourier restriction of $f$ to $N_{I,J}$.
\medskip

In order to prove \eqref{ww12} we need to  prove the following claim.
\begin{claim}
	The strip $\mc{M}_{1,I}$ lies within $O(K^{-1})$ from a cylinder like the one from Lemma \ref{ww8} (modulo rigid motions).
\end{claim}
\begin{proof}(of the claim)
	The tangent space $T_r$ to $\mc{M}_{1,I}$ at the point from $\B$ indexed by $r$ is spanned by the vectors
	$$v_1=(1,0,0,2Ar+Ds_0+Et)=(1,0,0,(2A+E\beta)r+Ds_0+E\gamma s_0 +E\alpha)$$
	$$v_2=(0,1,0,2Bs_0+Dr+Ft)=(0,1,0,(D+F\beta)r+2Bs_0+F\gamma s_0 +F\alpha)$$
	and
	$$v_3=(0,0,1,2Ct+Er+Fs_0)=(0,0,1,(E+2C\beta)r+Fs_0+2C\gamma s_0 +2C\alpha).$$
	
	Recall that $|a|=|A+C\beta ^2 +E\beta|>\eta /4.$ Therefore, by the triangle inequality, either $|2A+E\beta |>\eta/4$ or $|2C\beta^2 +E\beta|>\eta/4$. So we split into two cases here.
	\medskip
	
	First, assume $|2A+E\beta |>\eta/4$. Note that for each $r$, $T_r$ contains the fixed plane $\pi$ spanned by
	\beq
	\begin{split}
		u_3=& (-(D+F\beta),2A+E\beta,0,\\
		& -(D+F\beta)(Ds_0+E\gamma s_0 +E\alpha)+ (2A+E\beta)(2Bs_0+F\gamma s_0 +F\alpha))
	\end{split}
	\endeq
	and
	\beq
	\begin{split}
		& u_4= (-(E+2C\beta),0,2A+E\beta, \\
		& -(E+2C\beta)(Ds_0+E\gamma s_0 +E\alpha)+(2A+E\beta)(Fs_0 +2C\gamma s_0 +2C\alpha)).
	\end{split}
	\endeq
	Consider the cylinder $Cyl$ in $\R^4$ obtained by attaching the plane $\pi$ to each point of the parabola $\B$. In other words, $\pi$ will be the ``vertical component" of $Cyl$. In general, the plane $\pi$ is not perpendicular to the plane of the parabola. However, since
	\beq
	\label{fm2}|\det[w_1,w_2,u_3,u_4]|=
	|2A+E\beta||2A+2C\beta ^2 +2E\beta|>\eta/4 \times \eta/2
	\endeq
	is away from zero, the cylinder is non-degenerate. Its cross section with the plane $\pi^{\perp}$ is the projection of $\B$ onto $\pi^{\perp}$. Due to \eqref{ww11} and \eqref{fm2} this projection will be a curve given by $u_2=\psi(u_1)$, with $|\psi''|\sim 1$, for some appropriate orthonormal basis $(u_1,u_2)$ in $\pi^{\perp}$. In other words, $Cyl$ is a cylinder like the one in Lemma \ref{ww8}, modulo a rigid motion. Taylor's approximation of second order finishes the proof of the claim in this case, as $\mc{M}_{1,I}$ lies within $O(K^{-1/2})$ from $\B$.
	\medskip
	
	In the second case, assume $|2C\beta^2 +E\beta|>\eta/4$. The proof is similar to the first case, but this time we use
	\beq
	\begin{split}
		u_3=&\Big(E+2C\beta,0,-(2A+E\beta),\\
		& (E+2C\beta)(Ds_0+E\gamma s_0 +E\alpha)- (2A+E\beta)(Fs_0 +2C\gamma s_0 +2C\alpha)\Big)
	\end{split}
	\endeq
	and
	\beq
	\begin{split}
		& u_4= \Big(0,E\beta + 2C\beta ^2,-(D\beta+F\beta ^2), \\
		& (E\beta+2C\beta ^2)(2Bs_0+F\gamma s_0 +F\alpha)-(D\beta+F\beta ^2)(Fs_0 +2C\gamma s_0 +2C\alpha)\Big).
	\end{split}
	\endeq
	
\end{proof}

It follows that $\cup_JN_{I,J}$ lies in the $O(K^{-1})$-neighborhood of $Cyl$.
Let now $P_1,P_2,\ldots$ be the partition of this neighborhood like in Lemma \ref{ww8}. By choosing $P_i$ wide enough (still of order $O(K^{-1/2})$) we may arrange so that each $N_{I,J}$ is inside some $P_i$ and moreover, each $P_i$ contains at most one $N_{I,J}$. This can be seen via simple geometry, using the orientation of $Cyl$.

Thus, if $f$ is Fourier supported in $\cup_JN_{I,J}$, it is automatically Fourier supported in $\cup_iP_i$ and moreover
$$f_{N_{I,J}}=f_{P_i}$$
whenever $N_{I,J}\subset P_i$. With all these observations, inequality \eqref{ww12} is an immediate consequence of \eqref{ww13}. This finishes the analysis of Case 1.
\medskip

Case 2. Assume $|b|>\eta /4$. Then the proof is similar to the proof of Case 1 with the role of $r,s$ swapped.
\medskip

Case 3. Since we are not in Case 1 or 2, we may assume that $|a|\leq \eta /4, |b| \leq \eta/4, |c|>\eta$. Suppose $Q_1 (r,s,t)= Ar^2 +B s^2 +C t^2 +D rs + E rt + F st$ for some $A,B,C,D,E,F \in \R$. Then by a direct computation,
\beq
\label{c}
c=2\beta \gamma C+D+E\gamma + F\beta, \;\;a=A+ C\beta ^2 + E \beta ,\;\; b=B+C\gamma ^2+F \gamma.
\endeq

Our approach here is similar to what we did before, we will use a cylindrical decoupling. But this time, the base would be a two dimensional surface in $\R ^3$ with nonzero Gaussian curvature. \\\

Tile the unit square $\{(r,s)\in[0,1]\times [0,1]\}$ with $K^{1/2}-$squares, and call this collection $\mc{R}_{tile}$. By allowing another $O(1)$ loss, we may in fact assume that there is at most one $R\in \mc{R}$ whose $(r,s)-$projection is any given  square in $\mc{R}_{tile}$.
Let
$$h=g\sum_{R\in\mc{R}}1_R.$$
We will prove that for each $p\geq 4$,
\beq
\label{ww18}
\|E_{[0,1]^3}h\|_{L^p(w_{B_K})}\lesim_\epsilon K^{1-\frac3{p}+\epsilon} (\sum_{I,J\subset [0,1]\atop{|I|=|J|=K^{-1/2}}} \|E_{I\times J\times [0,1]}h\|_{L^p(w_{B_K})}^p)^{1/p}.
\endeq

By invoking Lemma \ref{abc8}, it will suffice to prove the following inequality for each $\widetilde{h}$ with $|\widetilde{h}|=|h|$
\beq
\label{ww15}
\|E_{[0,1]^3}^{(1)}\widetilde{h}\|_{L^p(w_{B_K})}\lesim_\epsilon K^{1-\frac3{p}+\epsilon} (\sum_{I,J\subset [0,1]\atop{|I|=|J|=K^{-1/2}}} \|E^{(1)}_{I\times J\times [0,1]}\widetilde{h}\|_{L^p(w_{B_K})}^p)^{1/p},
\endeq
where $E^{(1)}$ is the extension operator for the manifold
$$\mc{M}_1=\{(r,s,t,Q_1(r,s,t)),\;|t-(\alpha+\beta r+\gamma s)|\lesssim K^{-1/2})\}.$$

Note that $\mc{M}_1$ lies in the $O(K^{-1/2})$-neighborhood of the surface
$$\mathbb{S}=\{(r,s,\alpha+\beta r +\gamma s, ar^2+b s ^2+ crs):\;r,s\in[0,1]\},$$
whose Gaussian curvature satisfies
\beq
\label{ww16}
\kappa\sim 1,
\endeq
since $|c^2-4ab|\geq |c|^2-4 |a||b|\geq \eta ^2 -4 \eta/4 \times \eta /4=3\eta^2 /4$ is away from zero by assumption.
This surface lies in a translate of the three dimensional space spanned by
$$w_1=(1,0,\beta,0),$$
$$w_2=(0,1,\gamma,0),$$
and
$$w_3=(0,0,0,1).$$
Tile $\mc{M}_{1}$ with caps
$$\mc{M}_{1,I,J}=\{(r,s,t,Q_1(r,s,t)):\;(r,s,t)\in I\times J\times [0,1],\;|t-(\alpha+\beta r+\gamma s)|\lesssim K^{-1/2}\}.$$
Note that the Fourier transform of $E^{(1)}_{I\times J\times [0,1]}\widetilde{h}$ is supported on the cap $\mc{M}_{1,I,J}$. By loosing $O(1)$ we may assume that the $K^{-1}$-neighborhoods $N_{I,J}$ of these caps are pairwise disjoint. Thus (\ref{ww15}) will follow if we prove that for each $f$ Fourier supported in $\cup_{I,J}N_{I,J}$ we have
\beq
\label{ww17}
\|f\|_{L^p(w_{B_K})}\lesssim_{\epsilon} K^{1-\frac3{p}+\epsilon}(\sum_{I,J\subset [0,1]\atop{|I|=|J|=K^{-1/2}}}\|f_{N_{I,J}}\|_{L^p(w_{B_K})}^p)^{1/p},
\endeq
where $f_{N_{I,J}}$ is the Fourier restriction of $f$ to $N_{I,J}$.
\medskip

In order to prove (\ref{ww17}) we need to  prove the following claim.
\begin{claim}
	$\mc{M}_{1}$ lies within $O(K^{-1})$ from a cylinder like the one from Lemma \ref{cyl} (modulo rigid motions).
\end{claim}
\begin{proof}(of the claim)
	The tangent space $T_{r,s}$ to $\mc{M}_{1}$ at the point from $\mathbb{S}$ indexed by $r,s$ is spanned by the vectors
	$$v_1=(1,0,0,2Ar+Ds+Et)=(1,0,0,(2A+E\beta)r+(D+E\gamma) s +E\alpha)$$
	$$v_2=(0,1,0,2Bs+Dr+Ft)=(0,1,0,(D+F\beta)r+(2B+F\gamma) s +F\alpha)$$
	and
	$$v_3=(0,0,1,2Ct+Er+Fs)=(0,0,1,(E+2C\beta)r+(F+2C\gamma) s +2C\alpha).$$
	
	By a direct computation, we have the following identity:

	\beq
	\label{form1}
	\begin{split}
		&\Big|\det
		\begin{bmatrix}
			2A+E\beta & D+F\beta \\
			D+E\gamma  &  2B+F\gamma
		\end{bmatrix}
		-\beta
		\det
		\begin{bmatrix}
			D+F\beta & E+2C\beta\\
			2B+F\gamma & F+2C\gamma
		\end{bmatrix}\\
		&
		-\gamma
		\det
		\begin{bmatrix}
			D+E\gamma & F+2C\gamma\\
			2A+E\beta & E+2C\beta
		\end{bmatrix}\Big|\\
		&= |(2\beta \gamma C+D+E\gamma + F\beta) ^2-4(A+ C\beta ^2 + E \beta)(B+C\gamma ^2+F \gamma)|
	\end{split}
	\endeq
	
	Since the right hand side of the equation is equal to $|c^2- 4ab|$ which is away from $0$, at least one term from the left hand side of the equation must be away from $0$. In particular, this tells us that the
	rank of
	$$\begin{bmatrix}
	2A+E\beta & D+F\beta & E+2C\beta\\
	D+E\gamma & 2B+F\gamma & F+2C\gamma
	\end{bmatrix}$$
	is two.\\\
	
	From this, we could deduce that for each $r,s$, $T_{r,s}$ contains a line $L$ parallel to the vector
	\beq
	\begin{split}
		u_4= \Big(  \det
		\begin{bmatrix}
			D+F\beta & E+2C\beta\\
			2B+F\gamma & F+2C\gamma
		\end{bmatrix} ,
		& \det
		\begin{bmatrix}
			D+E\gamma & F+2C\gamma\\
			2A+E\beta & E+2C\beta
		\end{bmatrix},\\
		&    \det
		\begin{bmatrix}
			2A+E\beta & D+F\beta \\
			D+E\gamma  &  2B+F\gamma
		\end{bmatrix}, \ast\Big)
	\end{split}
	\endeq
	where
	\beq
	\begin{split}
		\ast &= \det
		\begin{bmatrix}
			D+F\beta & E+2C\beta\\
			2B+F\gamma & F+2C\gamma
		\end{bmatrix} \times ((2A+E\beta)r+(D+E\gamma) s +E\alpha)
		+ \\
		&\det
		\begin{bmatrix}
			D+E\gamma & F+2C\gamma\\
			2A+E\beta & E+2C\beta
		\end{bmatrix} \times ((D+F\beta)r+(2B+F\gamma) s +F\alpha)
		+ \\
		&\det
		\begin{bmatrix}
			2A+E\beta & D+F\beta \\
			D+E\gamma  &  2B+F\gamma
		\end{bmatrix} \times ((E+2C\beta)r+(F+2C\gamma) s +2C\alpha).
	\end{split}
	\endeq
	The main point is that $\ast$ is independent of $r,s$ (after simplification, the coefficient of $r,s$ is $0$), so that $u_4$ is independent of $r,s$. \\\

	Consider the cylinder $Cyl$ in $\R^4$ obtained by attaching the line $L$ parallel to $u_4$ to each point of the surface $\mathbb{S}$. In other words, $L$ will be the ``vertical component" of $Cyl$. In general, the line $L$ is not perpendicular to the three dimensional space where the surface lies. However, since
	\beq\label{nondeg}
	\begin{split}
		& |\det[w_1,w_2,w_3,u_4]|=
		\Big|\det
		\begin{bmatrix}
			2A+E\beta & D+F\beta \\
			D+E\gamma  &  2B+F\gamma
		\end{bmatrix}\\
		&   -\beta
		\det
		\begin{bmatrix}
			D+F\beta & E+2C\beta\\
			2B+F\gamma & F+2C\gamma
		\end{bmatrix}
		-\gamma
		\det
		\begin{bmatrix}
			D+E\gamma & F+2C\gamma\\
			2A+E\beta & E+2C\beta
		\end{bmatrix}\Big|
	\end{split}
	\endeq
	is away from zero by formula \eqref{form1}, the cylinder is non-degenerate. Its cross section with the space $L^{\perp}$ is the projection of $\mathbb{S}$ onto $L^{\perp}$. Due to \eqref{ww16} and \eqref{nondeg} this projection will be a surface given by $u_3=\psi(u_1,u_2)$, with $|D^2(\psi)|\sim 1$, for some appropriate orthonormal basis $(u_1,u_2,u_3)$ in $L^{\perp}$. In other words, $Cyl$ is a cylinder like the one in Lemma \ref{cyl}, modulo a rigid motion. Taylor's approximation of second order finishes the proof of the claim, as $\mc{M}_{1}$ lies within $O(K^{-1/2})$ from $\mathbb{S}$.\\\

\end{proof}

It follows that $\cup_{I,J}N_{I,J}$ lies in the $O(K^{-1})$-neighborhood of $Cyl$.
Let now $P_1,P_2,\ldots$ be the partition of this neighborhood like in Lemma \ref{cyl}. By choosing $P_i$ wide enough (still of order $O(K^{-1/2})$) we may arrange so that each $N_{I,J}$ is inside some $P_i$ and moreover, each $P_i$ contains at most one $N_{I,J}$. This can be seen via simple geometry, using the orientation of $Cyl$.

Thus, if $f$ is Fourier supported in $\cup_{I,J}N_{I,J}$, it is automatically Fourier supported in $\cup_iP_i$ and moreover
$$f_{N_{I,J}}=f_{P_i}$$
whenever $N_{I,J}\subset P_i$. With all these observations, inequality \eqref{ww17} is an immediate consequence of \eqref{ww14bju9fgjb89tibv90rt890u}.

Thus, since $1-\frac3p\le \frac32(\frac12-\frac1p)$ when $p\le 6$, we have that  \eqref{1903e1.39} is a consequence of \eqref{ww18}.

\bigskip

This ends the analysis of Case 3 and thus  the proof of Theorem \ref{1903lemma1.14}.

\section{From planes to arbitrary surfaces}
\label{snew}
Throughout this section we will fix $p\in(2,\infty)$ and will assume that the inequality
\beq\label{jj2}
\|E_Sg\|_{L^p(w_{B_M})}\lesssim M^{\gamma}(\sum_{Q\subset S\atop{l(Q)=M^{-1/2}}}\|E_Qg\|^p_{L^p(w_{B_M})})^{1/p}
\endeq
holds true for all $M\ge 1$ and for all rectangular boxes $S\subset [0,1]^3$ with size $\sim M^{-1/2}\times 1\times 1$. In our applications, we will take $p=\frac{14}{3}$.

The forthcoming discussion is following very closely the arguments from \cite{Oh}. This is a variant of the induction on scales that was used in \cite{PrSe}  and then in \cite{BD1} to prove the sharp decoupling for the cone. The intriguing aspect in the present context is that we approximate curved surfaces with zero curvature manifolds (planes). To bridge the gap between zero curvature and nonzero curvature we use the following rescaling argument.

\begin{lem}
	\label{jj1}
	For each rectangular box $R\subset [0,1]^3$ with size $\sim M^{-1}\times M^{-1/2}\times M^{-1/2}$ we have
	$$\|E_Rg\|_{L^p(w_{B_{M^2}})}\lesssim M^{\gamma}(\sum_{Q'\subset R\atop{l(Q')=M^{-1}}}\|E_{Q'}g\|^p_{L^p(w_{B_{M^2}})})^{1/p}.$$
\end{lem}
\begin{proof}
	The argument is a standard parabolic rescaling. Rescale $t,r,s$ by $M^{1/2}$. The ball $B_{M^2}$ from $\R^5$ will turn into a set that resembles a box with size $M^{3/2}\times M^{3/2}\times M^{3/2}\times M\times M$. Cover it with balls $B_M$, apply \eqref{jj2} on each $B_M$ then sum up all these contributions.
	
\end{proof}
The key observation is that \eqref{jj2} forces a similar inequality for curved boxes.

\begin{prop}
	\label{reiou rtyg8 we9d0-e r8gew=-d}
	The inequality
	\beq\label{jj3}
	\|E_Ug\|_{L^p(w_{B_K})}\lesssim_\epsilon K^{\gamma+\epsilon}(\sum_{Q\subset U\atop{l(Q)=K^{-1/2}}}\|E_Qg\|^p_{L^p(w_{B_K})})^{1/p}
	\endeq
	holds true for all $K\ge 1$ and $\epsilon>0$, where $U\subset [0,1]^3$ is the $K^{-1/2}$ neighborhood of a smooth surface in $\R^3$ (the graph of a smooth function). The implicit constant is uniform over surfaces with principal curvatures of magnitude $O(1)$.
\end{prop}
\begin{proof}
	Fix $\epsilon>0$ of the form $\epsilon=2^{-n-1}$ with $n\in \N$. We may assume that $g$ is supported on $U$.
	
	Cover $U$ with $\sim K^{2\epsilon}$ rectangular boxes $R_1$ of size $K^{-2\epsilon}\times K^{-\epsilon}\times K^{-\epsilon}$, then write using H\"older's inequality
	\beq\label{jj4}
	\|E_Ug\|_{L^p(w_{B_K})}\lesssim K^{2\epsilon(1-\frac1p)}(\sum_{R_1}\|E_{R_1}g\|^p_{L^p(w_{B_K})})^{1/p}.
	\endeq
	Next we apply Lemma \ref{jj1} with $M=K^{2\epsilon}$ on each ball $B_{K^{4\epsilon}}$ in a finitely overlapping cover of $B_K$ and then sum over these balls to get
	\beq\label{jj5}
	\|E_{R_1}g\|_{L^p(w_{B_K})}\le C K^{2\epsilon\gamma}(\sum_{Q_1\subset R_1\atop{l(Q_1)=K^{-2\epsilon}}}\|E_{Q_1}g\|^p_{L^p(w_{B_K})})^{1/p}.
	\endeq
	
	We repeat this argument as follows. Fix a $Q_1$ as above and note that $E_{Q_1}g=E_{Q_1\cap U}g$. Note also that $Q_1\cap U$ is contained in a rectangular box $R_2$ with size $\sim K^{-4\epsilon}\times K^{-2\epsilon}\times K^{-2\epsilon}$, and we may thus write
	$$E_{Q_1}g=E_{R_2}g.$$
	Apply Lemma \ref{jj1} as above with $M=K^{4\epsilon}$ to write
	\beq\label{jj6}
	\|E_{Q_1}g\|_{L^p(w_{B_K})}=\|E_{R_2}g\|_{L^p(w_{B_K})}\le C K^{4\epsilon\gamma}(\sum_{Q_2\subset R_2\atop{l(Q_1)=K^{-4\epsilon}}}\|E_{Q_2}g\|^p_{L^p(w_{B_K})})^{1/p}.
	\endeq
	
	We iterate this procedure. In the final step, we are faced with cubes $Q_{n-1}$ with side length $K^{-2^{n-1}\epsilon}=K^{-\frac14}$. Since $Q_{n-1}\cap U$ is inside a rectangular box $R_n$ with size $$\sim K^{-2^n\epsilon}\times K^{-2^{n-1}\epsilon}\times K^{-2^{n-1}\epsilon}=K^{-1/2}\times K^{-1/4}\times K^{-1/4}$$
	we may apply  Lemma \ref{jj1} one last time with $M=K^{1/2}$ to write
	\beq\label{jj7}
	\|E_{Q_{n-1}}g\|_{L^p(w_{B_K})}=\|E_{R_n}g\|_{L^p(w_{B_K})}\le C K^{\frac12\gamma}(\sum_{Q_n\subset R_n\atop{l(Q_n)=K^{-1/2}}}\|E_{Q_n}g\|^p_{L^p(w_{B_K})})^{1/p}.
	\endeq
	Collecting \eqref{jj4} through \eqref{jj7} we conclude that
	$$\|E_Ug\|_{L^p(w_{B_K})}\lesssim C^n K^{2\epsilon(1-\frac1p)}K^{\gamma(\frac12+\frac14+\ldots)}(\sum_{Q_n\subset U\atop{l(Q_n)=K^{-1/2}}}\|E_{Q_n}g\|^p_{L^p(w_{B_K})})^{1/p}$$
	which is equivalent to \eqref{jj3}.
\end{proof}

We can now prove the following consequence of Theorem \ref{1903lemma1.14}.

\begin{cor}\label{1903lemma1.141}
	Let $H$ be a 2-variety in $\R^3$ which intersects $[0, 1]^3$. Fix a large constant $K\gg 1$. Let $\mc{R}\subset Col_{K^{1/2}}$ be a collections of $K^{1/2}$-cubes, each of which intersects $H$. Then we have the decoupling inequality
	\beq\label{1903e1.39nnn}
	\|\sum_{\beta\in \mc{R}}E_{\beta}g\|_{L^p(w_{B_K})}\lesim_\epsilon K^{\frac{3}{2}(\frac{1}{2}-\frac{1}{p})+\epsilon} (\sum_{\beta\in \mc{R}} \|E_{\beta}g\|_{L^p(w_{B_K})}^p)^{1/p},
	\endeq
	for all $4\le p\le 6$.
\end{cor}
\begin{proof}
	Write $H$ as the union of $O(1)$ many manifolds of dimension at most two. It suffices to prove our inequality pretending $H$ is one of these manifolds. The case of zero dimension is trivial. If $H$ is one dimensional, the inequality follows from trivial decoupling (Lemma \ref{1903lemma1.12}), since the result from \cite{Wo} implies that $H$ intersects at most $O(K^{1/2})$ cubes from $Col_{K^{1/2}}$. Finally, if $H$ is a surface, we combine Theorem \ref{1903lemma1.14} with Proposition \ref{reiou rtyg8 we9d0-e r8gew=-d}.
\end{proof}

\section{Equivalence between linear and multilinear decoupling}
\label{s4}

In this subsection we run a version of the Bourgain-Guth argument from \cite{BG} to prove that the linear decoupling inequality \eqref{1803e1.6} is equivalent to a certain multilinear one. Recall that we work with a fixed $\mc{S}$ as in \eqref {ff9}. We continue to use the simplified notation $E$ to denote the extension operator $E^{\mc{S}}$, while $D(N,p)$ will refer to $D_{\mc{S}}(N,p)$. Define the multilinear  decoupling constant $D_{multi}(N, p, \nu)$ to be the smallest number such that
\beq
\label{ww22}
\|\prod_{i=1}^m |E_{R_i} g_i|^{1/m}\|_{L^p(w_{B_N})} \le D_{multi}(N, p, \nu) (\prod_{i=1}^m \sum_{\substack{\Delta\subset R_i\\l(\Delta)=N^{-1/2}}} \|E_{\Delta} g_i\|^p_{L^p(w_{B_N})})^{1/pm}.
\endeq
holds for all $\nu$-transverse cubes $R_i\subset [0,1]^3$ (both $m$ and the side lengths of the cubes can be arbitrary), all $g_i:R_i\to\C$ and all balls $B_N\subset \R^5$.
H\"older's inequality proves that
$$D_{multi}(N, p, \nu)\le D(N,p).$$

In the rest of this section, we will show that the reverse inequality is also essentially true. More precisely, we will prove the following result.
\begin{prop}\label{1903prop1.10}
	For each $K\gg 1$, $4\le p\le 6$ and $\epsilon>0$, there exists  $\beta(p, K, \epsilon)>0$ and $C(p,K)$ such that for each $\epsilon,p$ we have
	\beq
	\lim_{K\to \infty} \beta(p, K, \epsilon)=0,
	\endeq
	and  for each $N\ge K$ we have
	\beq\label{1803e1.31}
	\begin{split}
		& D(N, p)  \le N^{\beta(p, K, \epsilon)+\epsilon}  N^{\frac{3}{2}(\frac{1}{2}-\frac{1}{p})} \\
		&+ C(p,K) N^{\beta(p, K, \epsilon)+\epsilon}  \max_{1\le M\le N} \left[\Big(\frac{M}{N}\Big)^{-\frac{3}{2}(\frac{1}{2}-\frac{1}{p})}  D_{multi}(M, p, \nu_K) \right].
	\end{split}
	\endeq
\end{prop}

Here $\nu_K$ is the quantity appearing  in Theorem \ref{1803thm1.7}.
\begin{rem}
	\label{ww19}
	In light of the expected values for $D(N, p)$, see \eqref{1803e1.7}, the inequality \eqref{1803e1.31} shows that the value of $D(N, p)$ can not be significantly larger than that of  $D_{multi}(N, p, \nu_K)$, if $K$ is large enough.
\end{rem}

To prove the above proposition, we need several auxiliary lemmas. The first one is a ``trivial'' decoupling estimate. It makes use of the orthogonality among functions with frequencies supported on different caps, however it does not take advantage of the curvature of the surface $\mc{S}$ from \eqref{1803e1.1}.
\begin{lem}\label{1903lemma1.12}
	Let $R_1, ..., R_M$ be pairwise disjoint cubes in $[0, 1]^3$ with side length $K^{-1}$. Then for each $2\le p\le \infty$, we have
	\beq
	\|\sum_j E_{R_j} g\|_{L^p(w_{B_K})}\lesim_p M^{1-\frac{2}{p}}(\sum_j \|E_{R_j} g\|_{L^p(w_{B_K})}^p)^{1/p}
	\endeq
\end{lem}
\begin{proof}
	When $p=2$, we use the fact that the functions $E_{R_j}g$ have essentially disjoint frequency supports. At $p=\infty$, we use the triangle inequality. The rest  follows  from interpolation. See the proof of Lemma 5.1 from \cite{BD5} for details.
	
\end{proof}

Now we are ready to start the proof of Proposition \ref{1903prop1.10}. The main step is the proof of the following result.
\begin{prop}
	\label{r89ngyr7g7w8rt-02340rogm90ut9g0i}
	For each $4\le p\le 6$, each $\epsilon>0$ and $N\ge K\gg 1,$ we have that
	\beq\label{1903e1.40}
	\begin{split}
		& \|E_{[0, 1]^3} g\|_{L^p(\w)} \lesim_{\epsilon}  K^{\frac{3}{2}(1-\frac{2}{p})+\epsilon} (\sum_{R\in Col_K}\|E_R g\|_{L^p(\w)}^p)^{1/p}\\
		& + K^{\frac{3}{2}(\frac{1}{2}-\frac{1}{p})+\epsilon} \left( \sum_{\beta\in Col_{K^{1/2}}}\|E_{\beta}g\|_{L^p(\w)}^p \right)^{1/p} \\
		& + C(p,K) D_{multi}(N, p, \nu_K) \left(  \sum_{\Delta\in Col_{N^{1/2}}} \|E_{\Delta} g\|_{L^p(\w)}^p \right)^{1/p}.
	\end{split}
	\endeq
	Here $\nu_K$ is the quantity appearing  in Theorem \ref{1803thm1.7}, and $C(p,K)$ is a constant depending on $p$.
\end{prop}
\begin{proof}
	Partition $[0, 1]^3$ into cubes $R$ from $Col_K$. Following Bourgain and Guth \cite{BG}, we may assume that $|E_R(x)|$ is essentially constant on each ball $B_K$ of radius $K$. This value will be denoted as $|E_R g(B_K)|$. Write
	\beq
	E_{[0, 1]^3} g(x)=\sum_{R\in Col_K} E_R g(x),\;\;x\in B_K.
	\endeq
	For a fixed $B_K$, let $R^*\in Col_{K}$ be the cube that maximizes $|E_R g(B_K)|$. Let $Col_{K}^*$ be those cubes $R\in Col_K$ such that
	\beq
	|E_R g(B_K)|\ge K^{-3}|E_{R^*} g(B_K)|.
	\endeq
	Before we proceed, let us first explain the ideas. We will deal with three cases. The first case is when $Col_K^*$ contains a ``small'' amount of cubes. In this case,  applying only the triangle and the Cauchy--Schwarz inequality will suffice. The second case is when the cardinality of $Col^*_K$ is large, but the cubes in $Col^*_K$ are not clustered near any 2-variety in $\R^3$. By Theorem \ref{1803thm1.7}, we know that these cubes are transverse, which allows us to invoke multilinear estimates. The last case is when a big percentage of the cubes in $Col^*_K$ intersect a 2-variety in $\R^3$. In this case, we will rely on a lower dimensional decoupling inequality, that is \eqref{1903e1.39} from Theorem \ref{1903lemma1.14}.\\
	
	{\bf Case 1:} Suppose
	\beq
	\#(Col_K^*)<10^4.
	\endeq

	In this case we combine the triangle and the Cauchy--Schwarz inequality, to get a very favorable estimate. First we observe that for $x\in B_K$
	$$
	|E_{[0, 1]^3} g(x)| \le |\sum_{R\in Col_{K}^*} E_R g(x)|+|E_{R^*}g(B_K)| \lesssim (\sum_{R\in Col_{K}} |E_R g(x)|^p)^{1/p}
	$$
	Integrating on $B_K$  we get
	\beq\label{1903e1.45}
	\|E_{[0, 1]^3} g\|_{L^p(w_{B_K})} \lesim  (\sum_{R\in Col_K}\|E_R g\|_{L^p(w_{B_K})}^p)^{1/p}.
	\endeq
	Note that we get a better estimate than needed in this case.
	%we are no longer allowed to do trivial decoupling, instead we turn to lower dimensional decoupling. We look at hyperplanes in $\R^3$.\\
	\medskip
	
	{\bf Case 2:} Assume
	\beq
	m:=\#(Col_K^*)\ge 10^4.
	\endeq
	Moreover, assume  there does not exist any 2-variety in $\R^3$ whose $10/K$ neighbourhood intersects more than $m/100$ of the cubes from $Col_K^*$. Then by Theorem \ref{1803thm1.7}, these $m$ cubes are $\nu_K$-transverse. We may write
	\beq\label{1903e1.47}
	\begin{split}
		\|E_{[0, 1]^3} g\|_{L^p(w_{B_K})} & \lesim K^6 \max_{R_1, ..., R_m\atop{\nu_K-\text{transverse}}} \|\prod_{i=1}^m|E_{R_i} g|^{1/m}\|_{L^p(w_{B_K})}\\
		& \lesim K^6 \left( \sum_{R_1, ..., R_m\atop{\nu_K-\text{transverse}}} \|\prod_{i=1}^m|E_{R_i} g|^{1/m}\|^p_{L^p(w_{B_K})}\right)^{1/p}.
	\end{split}
	\endeq
	
	{\bf Case 3:} Suppose that  there is a 2-variety in $\R^3$ whose $10/K$ neighbourhood intersects more than $m/100$ of the (at least $10^4$) cubes from $Col_K^*$. Call this 2-variety $H_1$. Consider the $10K^{-1/2}$-neighbourhood $\mc{N}_{K^{-1/2}}(H_1)$ of $H_1$. Denote
	\beq
	Col_K^{(1)}:=Col_K^*\setminus \{R\in Col_K^*: R\subset \mc{N}_{K^{-1/2}}(H_1)\}.
	\endeq
	Moreover, define
	\beq
	m_1=\#\Big(Col_K^{(1)} \Big).
	\endeq
	We cover $\mc{N}_{K^{-1/2}}(H_1)$ using cubes $\beta$ from $Col_{K^{1/2}}$. By Corollary \ref{1903lemma1.141}
	\beq\label{1903e1.50}
	\left\|\sum_{\substack{\beta\in Col_{K^{1/2}}\\ \beta\subset \mc{N}_{K^{-1/2}}(H_1)}}E_{\beta}g\right\|_{L^p(w_{B_K})} \lesim_\epsilon K^{\frac{3}{2}(\frac{1}{2}-\frac{1}{p})+\epsilon} \left(\sum_{\substack{\beta\in Col_{K^{1/2}}\\ \beta\subset \mc{N}_{K^{-1/2}}(H_1)}} \|E_{\beta}g\|_{L^p(w_{B_K})}^p\right)^{1/p}.
	\endeq
	This takes care of the cubes inside $\mc{N}_{K^{-1/2}}(H_1)$. For cubes outside, we repeat the whole procedure, with $Col_K^*$ replaced by $Col_K^{(1)}$ and $m$ by $m_1$. This procedure will terminate in at most $\log K$ many steps, as $Col_K^{(i+1)}$ is at least one percent smaller than $Col_K^{(i)}$. The $\log K$ will be harmlessly absorbed into the $K^\epsilon$ term.\\
	
	We collect all the contributions of the type \eqref{1903e1.45}, \eqref{1903e1.47} and \eqref{1903e1.50} from each step,
	
	\beq
	\begin{split}
		& \|E_{[0, 1]^3} g\|_{L^p(w_{B_K})} \lesim_{\epsilon}  (\sum_{R\in Col_K}\|E_R g\|_{L^p(w_{B_K})}^p)^{1/p}\\
		& + K^{\frac{3}{2}(\frac12-\frac{1}{p})+\epsilon} \left( \sum_{\beta\in Col_{K^{1/2}}}\|E_{\beta}g\|_{L^p(w_{B_K})}^p \right)^{1/p} \\
		& + K^{6}  \left( \sum_{K\lesssim m\lesssim K^3}\sum_{R_1, ..., R_{m}: \nu_K transverse} \|\prod_{i=1}^{m}|E_{R_i} g|^{1/m}\|^p_{L^p(w_{B_K})}\right)^{1/p}.
	\end{split}
	\endeq
	Raising to the $p$-th power and summing over $B_K\subset B_N$, we obtain
	\beq
	\label{fuyqerf7wer7f8er789er89}
	\begin{split}
		& \|E_{[0, 1]^3} g\|_{L^p(\w)} \lesim_{\epsilon}   (\sum_{R\in Col_K}\|E_R g\|_{L^p(\w)}^p)^{1/p}\\
		& + K^{\frac{3}{2}(\frac12-\frac{1}{p})+\epsilon} \left( \sum_{\beta\in Col_{K^{1/2}}}\|E_{\beta}g\|_{L^p(\w)}^p \right)^{1/p} \\
		& + K^{6}  \left( \sum_{K\lesssim m\lesssim K^3}\sum_{R_1, ..., R_{m}: \nu_K transverse} \|\prod_{i=1}^{m}|E_{R_i} g|^{1/m}\|^p_{L^p(\w)}\right)^{1/p}.
	\end{split}
	\endeq
	Note that there are only $O_K(1)$ choices of squares.
	By the definition of the multilinear decoupling constant in \eqref{ww22}, we conclude \eqref{1903e1.40}, as desired. Note that the first term in \eqref{fuyqerf7wer7f8er789er89} has a more favorable estimate than the one stated in \eqref{1903e1.40}. We prefer to work with the latter estimate, as it makes the rest of the argument more symmetric.
	
\end{proof}
Given a cube $R\subset [0,1]^3$ and $\alpha^{-1}<l(R)$, we denote by $Col_\alpha(R)$ the collection of all dyadic cubes inside $R$ with side length $\frac1\alpha$.

A standard rescaling gives (see Proposition 5.6 in \cite{BD1} for details)
\begin{prop}
	Let $R\subset [0, 1]^3$ be a cube with side length $\delta$. Then for each $\epsilon>0$, $K\ge 1$ and $N>\delta^{-2}$, we have
	\beq\label{2012e1.38}
	\begin{split}
		& \|E_{R} g\|^p_{L^p(\w)} \le C_{p,\epsilon}\big[  K^{\frac{3}{2}(p-2)+\epsilon} \sum_{R'\in Col_{K/\delta}(R)}\|E_{R'} g\|_{L^p(\w)}^p\\
		& + K^{\frac{3}{2}(\frac{p}{2}-1)+\epsilon}  \sum_{\beta\in Col_{K^{1/2}/\delta}(R)}\|E_{\beta}g\|_{L^p(\w)}^p  \\
		& + C(p,K) D^p_{multi}(N\delta^2, p, \nu_K)  \sum_{\Delta\in Col_{N^{1/2}}(R)} \|E_{\Delta} g\|_{L^p(\w)}^p\big].
	\end{split}
	\endeq
\end{prop}

We have arrived at the final stage of the proof of Proposition \ref{1903prop1.10}. We iterate the above result, from scale one, until scale $K^n$ is reached, where $n$ is such that
\beq
K^n= N^{1/2}.
\endeq
In other words, the iteration of each term terminates exactly when it equals  the last term in \eqref{2012e1.38}. At the end of the iteration, we will get many copies of the last term in \eqref{2012e1.38}, each of which comes with a certain coefficient.  Let us trace the iteration history of such a term. Suppose that throughout the iteration history the scale gets smaller by a factor of $\delta$ exactly $\lambda_1$ times and by a factor of $\delta^{1/2}$ exactly $\lambda_2$ times. Then
\beq
\lambda_1+\frac{\lambda_2}{2}\le n=\frac{1}{2}\log_K N.
\endeq
The corresponding coefficient of the final term corresponding to this $(\lambda_1,\lambda_2)$ pattern of iterations is
\beq
(C_{p, \epsilon})^{\lambda_1+\frac{\lambda_2}{2}}K^{(\lambda_1+\frac{\lambda_2}{2})\cdot \epsilon}K^{\frac{3}{2}(p-2)\lambda_1+\frac{3}{2}(\frac{p}{2}-1)\lambda_2}
\endeq
Notice that
$$
K^{(\lambda_1+\frac{\lambda_2}{2})\cdot \epsilon} \le K^{\epsilon\cdot \log_K N}\le N^{\epsilon},
$$
$$
(C_{p, \epsilon})^{\lambda_1+\frac{\lambda_2}{2}} \le N^{\log_K C_{p, \epsilon}},
$$
$$
K^{\frac{3}{2}(p-2)\lambda_1+\frac{3}{2}(\frac{p}{2}-1)\lambda_2}\le N^{\frac32(\frac{p}{2}-1)}.
$$
It is easy to see that there are at most $2^n$ terms corresponding to a given $(\lambda_1,\lambda_2)$ pattern. We write $2^n=N^{\log_K2}$.
Hence we obtain
$$D(N, p)^p  \le N^{\epsilon+\log_K C_{p, \epsilon}}  N^{\log_K2}\sum_{\lambda_1+(\lambda_2/2)=\frac{1}{2}\log_K N}  K^{\frac{3}{2}(p-2)\lambda_1+\frac{3}{2}(\frac{p}{2}-1)\lambda_2} \\
$$
\begin{multline*}
	+ C(p,K) N^{\epsilon+\log_K C_{p, \epsilon}}  N^{\log_K2}\\
	\sum_{\lambda_1+(\lambda_2/2)<\frac{1}{2}\log_K N} K^{\frac{3}{2}(p-2)\lambda_1+\frac{3}{2}(\frac{p}{2}-1)\lambda_2}  D^p_{multi}(NK^{-2\lambda_1-\lambda_2}, p, \nu_K)
\end{multline*}
\begin{multline*}
	\le N^{\epsilon+\log_K (2C_{p, \epsilon})}\log_K N \\
	[N^{\frac{3}{4}(p-2)}
	+ C(p,K) \sum_{j<\log_K N} K^{\frac{3j}{2}(\frac{p}2-1)} D^p_{multi}(NK^{-j}, p, \nu_K)]\
\end{multline*}
\begin{multline*}
	\le N^{\epsilon+\log_K (2C_{p, \epsilon})}(\log_K N)^2 \\
	[N^{\frac{3}{4}(p-2)}
	+ {C(p,K)} \max_{j<\log_K N} K^{\frac{3j}{2}(\frac{p}2-1)} D^p_{multi}(NK^{-j}, p, \nu_K)].
\end{multline*}
This finishes the proof of Proposition \ref{1903prop1.10}, using
$$\beta(p,K,\epsilon)=\frac1{p}\log_K (2C_{p, \epsilon}).$$

\section{The final iteration}
\label{s5}

In this section we finish the proof of \eqref{1803e1.7}. The argument here is entirely standard, and it appears in all recent papers related to decouplings.
\medskip

Let $4\le p\le 6$.  Fix $K\gg 1$ and $\nu_K$ transverse cubes $R_1,\ldots,R_m\subset [0,1]^3$. Fix also $g_i:R_i\to\C$.
Combining the inequality in Proposition \ref{iovjurgyptn8vbgu89357893v7589ty7056893} with H\"older's inequality we derive the following critical inequality, valid for
$\kappa_p=\frac{p-\frac{10}{3}}{p-2}\le \kappa\le 1$
$$\|(\prod_{i=1}^{m}\sum_{\atop{l(\tau)=N^{-1/4}}}|E_{\tau}g_i|^2)^{\frac1{2m}}\|_{L^{p}(w_{B_R})}\lesssim_{\epsilon,K}$$
\beq
\label{2003e1.71}
N^{\epsilon+\frac{3\kappa}4(\frac12-\frac1p)}\|(\prod_{i=1}^{m}\sum_{\atop{l(\Delta)=N^{-1/2}}}|E_{\Delta}g_i|^2)^{\frac1{2m}}
\|_{L^{p}(w_{B_R})}^{1-\kappa}
(\prod_{i=1}^{m}\sum_{\atop{l(\tau)=N^{-1/4}}}\|E_{\tau}g_i\|_{L^{p}(w_{B_R})}^p)^{\frac{\kappa}{pm}}.
\endeq

By the Cauchy-Schwarz inequality, we have for $s\ge 1$
\beq\label{2003e1.72}
\|(\prod_{j=1}^m |E_{R_j} g_j|)^{1/m}\|_{L^p(\w)} \le N^{\frac{3}{2}\cdot 2^{-s}} \|(\prod_{j=1}^m \sum_{l(\tau)=N^{-2^{-s}}}|E_{\tau} g_j|^2)^{1/2m}\|_{L^p(\w)}.
\endeq

We start with \eqref{2003e1.72}, and apply the estimate \eqref{2003e1.71} until we reach the scale $N^{-1/2}$. We control the last term in \eqref{2003e1.71} that appears in each step of the  iteration by parabolic rescaling. That is, for each cube $R\subset [0, 1]^3$ with side length $L$, we have
\beq
\|E_R g\|_{L^p(\w)}\le D(\frac{N}{L^2}, p) (\sum_{\substack{\Delta\subset R\\ l(\Delta)=N^{-1/2}}} \|E_{\Delta} g\|_{L^p(\w)}^p )^{1/p}.
\endeq
In the end, we obtain
\beq
\begin{split}
	& \|(\prod_{j=1}^m |E_{R_j} g_j|)^{1/m}\|_{L^p(\w)} \\
	& \le N^{\frac{3}{2} \cdot 2^{-s}} (C_{p, K, \epsilon} N^{\epsilon})^{s-1} N^{\frac{3\kappa}{4}(\frac{1}{2}-\frac{1}{p})(1-\kappa)^{s-2}}\times ... \times N^{\frac{3\kappa}{2^{s-1}}(\frac{1}{2}-\frac{1}{p})(1-\kappa)} \times N^{\frac{3\kappa}{2^s}(\frac{1}{2}-\frac{1}{p})} \\
	&  \times D(N^{1-2^{-s+1}}, p)^{\kappa} \times D(N^{1-2^{-s+2}}, p)^{\kappa(1-\kappa)}\times  ... \times D(N^{1/2}, p)^{\kappa(1-\kappa)^{s-2}}\\
	& \|(\prod_{j=1}^m \sum_{l(\Delta)=N^{-1/2}}|E_{\Delta} g_j|^2)^{1/2m}\|^{(1-\kappa)^{s-1}}_{L^p(\w)} \big(\prod_{j=1}^m  \sum_{l(\Delta)=N^{-1/2}} \|E_{\Delta}g_j\|_{L^{p}(\w)}^p\big)^{\frac{1-(1-\kappa)^{s-1}}{pm}}
\end{split}
\endeq
By H\"older's and Minkowski's inequality, we bound the second to last term by
\beq
N^{\frac{3}{2}(\frac{1}{2}-\frac{1}{p})(1-\kappa)^{s-1}} \left(\prod_{j=1}^m  \sum_{l(\Delta)=N^{-1/2}} \|E_{\Delta}g_j\|_{L^{p}(\w)}^p\right)^{\frac{(1-\kappa)^{s-1}}{pm}}.
\endeq
By taking supremum over all $\nu_K$- transverse cubes $R_i$ and all $g_i:R_i\to\C$, these observations lead to
\beq\label{2003e1.76}
\begin{split}
	& D_{multi}(N, p, \nu_K) \le (C_{p, K, \epsilon} N^{\epsilon})^{s-1} N^{\frac{3}{2}\cdot 2^{-s}+ \frac{3}{2}(\frac{1}{2}-\frac{1}{p})(1-\kappa)^{s-1}} N^{3\kappa 2^{-s}(\frac{1}{2}-\frac{1}{p})\frac{1-(2(1-\kappa))^{s-1}}{2\kappa-1}}\\
	& D(N^{1-2^{-s+1}}, p)^{\kappa} \times D(N^{1-2^{-s+2}}, p)^{\kappa(1-\kappa)}\times  ... \times D(N^{1/2}, p)^{\kappa(1-\kappa)^{s-2}}.
\end{split}
\endeq
\bigskip

Now we come to the final step of the proof Theorem \ref{1803thm1.1}. Recall that  we have shown that the linear decoupling constant $D(N, p)$ is essentially controlled by the multilinear decoupling constant $D_{multi}(N, p, \nu_K)$. See the estimate \eqref{1803e1.31} from Proposition \ref{1903prop1.10}. The estimate \eqref{2003e1.76} also reveals a connection between these two constants. We will see that these two estimates together lead to the final conclusion.
\medskip

Let $\gamma_p$ be the unique positive constant such that
\beq
\lim_{N\to \infty} \frac{D(N, p)}{N^{\gamma_p+\delta}}=0, \text{ for each } \delta>0,
\endeq
and
\beq
\limsup_{N\to \infty} \frac{D(N, p)}{N^{\gamma_p-\delta}}=\infty, \text{ for each } \delta>0.
\endeq

By substituting the estimate $D(N, p)\lesim_{\delta} N^{\gamma_p+\delta}$ into \eqref{2003e1.76}, we obtain
\beq\label{2003e1.79}
\limsup_{N\to \infty} \frac{D_{multi(N, p, \nu_K)}}{N^{\gamma_{\kappa, p, \delta, s, \epsilon}}}<\infty.
\endeq
Here
\beq\label{2003e1.80}
\begin{split}
	\gamma_{\kappa, p, \delta, s, \epsilon}=& \epsilon(s-1)+\frac{3}{2}\cdot 2^{-s} +\frac{3}{2}(\frac{1}{2}-\frac{1}{p})(1-\kappa)^{s-1}\\
	&+ 3\kappa 2^{-s}(\frac{1}{2}-\frac{1}{p})\frac{1-(2(1-\kappa))^{s-1}}{2\kappa-1}\\
	& + \kappa(\gamma_p+\delta)\Big( \frac{1-(1-\kappa)^{s-1}}{\kappa} -2^{-s+1} \frac{1-(2(1-\kappa))^{s-1}}{2\kappa-1}\Big).
\end{split}
\endeq
By invoking interpolation, it suffices to prove that
$$\gamma_{\frac{14}{3}}\le \frac{3}{2}(\frac12-\frac{3}{14})=\frac37$$
We observe that
\beq
2(1-\kappa_{\frac{14}{3}})=1.
\endeq
This is precisely the relation that shows that $\frac{14}{3}$ is the critical exponent for our decoupling.
It suffices to prove that for each $\kappa>\frac12$
\beq
\gamma_{\frac{14}{3}}\le \frac{3}{2}(\frac{2\kappa-1}{2\kappa}+\frac{1}{2}-\frac3{14}).
\endeq
Assume for contradiction that there exists $\kappa>\frac12$ such that
\beq\label{2003e1.83}
\gamma_{\frac{14}{3}}> \frac{3}{2}(\frac{2\kappa-1}{2\kappa}+\frac{1}{2}-\frac{3}{14}).
\endeq
Using \eqref{2003e1.83}, multiplying both side of \eqref{2003e1.80} by $2^{s}$, letting $s$ be large enough, and then $\epsilon$ and $\delta$ be small enough, we obtain
\beq
\gamma_{\kappa, \frac{14}{3}, \delta, s, \epsilon}< \gamma_{\frac{14}3}.
\endeq

Fix small enough $\epsilon, \delta$ and large enough $s$, then  choose $K$ so large that
\beq\label{2003e1.86}
\frac{3}{2} (\frac{1}{2}-\frac{3}{14})+\epsilon+\beta(\frac{14}{3}, K,\epsilon)< \frac{3}{2}(\frac{2\kappa-1}{2\kappa}+\frac{1}{2}-\frac{3}{14}),
\endeq
and
\beq\label{2003f1.87}
\gamma_{\kappa, \frac{14}{3}, \delta, s, \epsilon}+\epsilon+\beta(\frac{14}{3}, K,\epsilon)< \gamma_\frac{14}{3}.
\endeq
Here $\beta(\frac{14}{3}, K,\epsilon)$ is the constant that appears in Proposition \ref{1903prop1.10}. Now combining Proposition \ref{1903prop1.10} with \eqref{2003e1.83} and  \eqref{2003e1.86}, we find that
\beq\label{2003e1.87}
D(N,\frac{14}{3} )\lesim_{K, \epsilon} N^{\epsilon+\beta(\frac{14}{3}, K, \epsilon)} \max_{1\le M\le N} \left[\Big(\frac{M}{N}\Big)^{-\frac{3}{2}(\frac{1}{2}-\frac{3}{14})}  D_{multi}(M, \frac{14}{3}, \nu_K) \right].
\endeq
We distinguish two cases, each of which will lead to a contradiction. \smallskip

Case 1. Assume  $\gamma_{\kappa, \frac{14}{3}, \delta, s, \epsilon}<\frac{3}{2}(\frac{1}{2}-\frac{3}{14})$. Then by \eqref{2003e1.79} and \eqref{2003e1.87}, we obtain
\beq
D(N, \frac{14}{3})\lesim_{K, \epsilon} N^{\epsilon+\beta(\frac{14}{3}, K, \epsilon)} N^{\frac{3}{2}(\frac{1}{2}-\frac{3}{14})}.
\endeq
By \eqref{2003e1.86}, this contradicts the assumption \eqref{2003e1.83}.
\smallskip

Case 2. Assume  $\gamma_{\kappa, \frac{14}{3}, \delta, s, \epsilon}\ge \frac{3}{2}(\frac{1}{2}-\frac{3}{14})$.
Substitute this into \eqref{2003e1.87}, we obtain
\beq
D(N,\frac{14}{3} )\lesim_{K, \epsilon} N^{\epsilon+\beta(\frac{14}{3}, K, \epsilon)} N^{\gamma_{kappa,\frac{14}3, \delta, s, \epsilon}}.
\endeq
By \eqref{2003f1.87}, this contradicts the definition of the constant $\gamma_{\frac{14}{3}}$.

The analysis of these cases  shows that \eqref{2003e1.83} can not be true. This finishes the proof of the estimate $\gamma_{\frac{14}3}\le 3/7$, and thus, of Theorem \ref{1803thm1.1}.

\section{Some linear algebra}
\label{s6}

Let us start by recalling some notation. We are concerned with the  surface
$$\mc{S}:=\{(r, s, t, Q_1(r,s,t),Q_2(r,s,t)): (r, s, t)\in [0, 1]^3\}$$
satisfying the requirement of Theorem \ref{1803thm1.1}. In this section, we will say that a property $\Omega=\Omega(\xi)$ (here $\xi \in \R ^3$) holds almost surely if $\{\xi: \Omega(\xi) \text{ does not hold}\}\neq \R ^3$. We will write $V_{r,s,t}$ to denote the tangent space to $\mc{S}$ at $(r, s, t, Q_1(r,s,t),Q_2(r,s,t))$  and $\pi_{r,s,t}$ to denote the orthogonal projection onto it.

The main purpose of this section is to prove the following lemma.

\begin{lem}\label{2201lemma3.1}
	1) Let $V$ be a one  dimensional linear subspace of $\R^5$. Then the set
	\beq\label{1903e1.70}
	\{(r, s, t): dim(\pi_{r, s, t}(V))=0\}
	\endeq
	is contained in a 2-variety. \\
	2)  Let $V$ be a two (resp. four) dimensional linear subspace of $\R^5$. Then the set
	\beq\label{guo8.2}
	\{(r, s, t): dim(\pi_{r, s, t}(V))\le 1  (\text{ resp. } 2)\}
	\endeq
	is contained in a 2-variety.\\
\end{lem}

\begin{proof}
	First, we will observe one consequence of the condition imposed in Theorem \ref{1803thm1.1}. Take $(u,v,w)=(1,0,0), (0,1,0)$ and $(0,0,1)$, we see that $$\det
	\begin{bmatrix}
	\partial_s Q_1 & \partial_t Q_1 \\
	\partial_s Q_2 & \partial_t Q_2
	\end{bmatrix} ,
	\det
	\begin{bmatrix}
	\partial_r Q_1 & \partial_t Q_1 \\
	\partial_r Q_2 & \partial_t Q_2
	\end{bmatrix},
	\det
	\begin{bmatrix}
	\partial_r Q_1 & \partial_s Q_1 \\
	\partial_r Q_2 & \partial_s Q_2
	\end{bmatrix}$$ are nonzero polynomials in $r,s,t$. (here $\partial_r Q_i$ is a shorthand for $\frac{\partial Q_i}{\partial r}$ and similar for $\partial_s Q_i$ and $\partial_t Q_i$) Thus in particular,
	\beq
	\label{lab2}
	\text{rank}
	\begin{bmatrix}
		\partial_r Q_1(\xi) & \partial_s Q_1(\xi) & \partial_t Q_1(\xi) \\
		\partial_r Q_2(\xi) & \partial_s Q_2(\xi) & \partial_t Q_2(\xi)
	\end{bmatrix}=2
	\endeq
	almost surely. \\\
	
	We start by proving the first statement. Notice that $V_{r, s, t}$ is given by the span of the three vectors
	\beq
	\begin{split}
		& n_1=(1, 0, 0, \partial_r Q_1, \partial_r Q_2),\\
		& n_2=(0, 1, 0, \partial_s Q_1, \partial_s Q_2),\\
		& n_3=(0, 0, 1, \partial_t Q_1, \partial_t Q_2).
	\end{split}
	\endeq
	Let $V\subset \R^5$ be a one-dimensional subspace. Suppose that $V=\text{span}\{x\}$ for some non-zero vector $x=(x_1,x_2,x_3,x_4,x_5)\in \R^5$. The dimension of $\pi_{r, s, t}(V)$ is equal to the rank of the matrix
	\beq\label{guo8.4}
	[x\cdot n_1, x\cdot n_2, x\cdot n_3].
	\endeq
	Moreover, if we view $x\cdot n_i$ with $i=1, 2, 3$ as affine functions in $r, s$ and $t$, we will show that at least one of them does not vanish constantly. Suppose this is not the case. Then $x_1 =x_2 =x_3 =0$ and
	\beq
	\label{lab1}
	x_4 \partial_r Q_1 + x_5 \partial_r Q_2= x_4 \partial_s Q_1+ x_5 \partial_s Q_2= x_4 \partial_t Q_1 + x_5 \partial_t Q_2\equiv 0 .
	\endeq
	Since $x$ is a nonzero vector, $(x_4, x_5)\neq (0,0)$. Hence by \eqref{lab1},
	$$\text{rank}
	\begin{bmatrix}
	\partial_r Q_1(\xi) & \partial_s Q_1(\xi) & \partial_t Q_1(\xi) \\
	\partial_r Q_2(\xi) & \partial_s Q_2(\xi) & \partial_t Q_2(\xi)
	\end{bmatrix}\leq 1.$$
	for every $\xi$.
	This contradicts \eqref{lab2}.\\\
	
	We turn to the proof of the second statement. The following approach is in the spirit of \cite{BDGu}. Define the vector spaces of polynomials
	\beq
	S_0=[1], S_1=[r, s, t] \text{ and } S_2=[Q_1(r, s, t), Q_2(r, s, t)].
	\endeq
	For $\xi=(r_0, s_0, t_0)\in \R^3$, let
	\beq
	P_{\xi}f(r, s, t)=f(\xi)+\partial_r f(\xi) (r-r_0)+\partial_s f(\xi) (s-s_0)+\partial_t f(\xi) (t-t_0)
	\endeq
	be the first order Taylor expansion of the function $f$ at the point $\xi$. Hence $P_{\xi}$ is a projection onto $S_0\oplus S_1$. Moreover, we have
	\beq
	\pi_{S_1} P_{\xi} f(r, s, t)=\partial_r f(\xi) r+\partial_s f(\xi) s+\partial_t f(\xi) t.
	\endeq
	Define $S=S_1\oplus S_2$. Let $V$ be a subspace of $\R ^5$. We could think of $V$ as a subspace of same dimension in $S$ by defining the isomorphism $(x_1,x_2,x_3,x_4,x_5)\mapsto (x_1 r+ x_2 s+ x_3 t +x_4 Q_1 + x_5 Q_2)$ from $\R^5$ to $S$. Under this correspondence, it is easy to see that $\text{dim}(\pi_{\xi}(V))= \text{dim}(\pi_{S_1}P_{\xi}(V))$, where $\pi_{\xi}(V)$ is the projection of $V$ onto the tangent space to $\mc{S}$ at $\xi$ when $V$ is considered as a subspace of $\R ^5$. Thus, we need to prove that almost surely in $\xi$,
	\beq
	\text{dim}(\pi_{S_1}P_{\xi}(V))=\text{dim}[(\partial_r f(\xi), \partial_s f(\xi), \partial_t f(\xi)): f\in V] \ge
	\begin{cases}
		2, \text{ if } \text{dim}(V)=2 \\
		3, \text{ if } \text{dim}(V)=4
	\end{cases}
	\endeq
	This will imply that the set \eqref{guo8.2} is contained in a 2-variety because the ``bad" set where the dimension is smaller than what we need is contained in the zero set of some nonzero polynomial of degree at most $2$.\\
	
	We first consider the case $\text{dim}(V)=2$. By contradiction, we assume that
	\beq\label{guo8.9}
	\text{dim}(\pi_{S_1}P_{\xi}(V))\le 1 \text{ for every } \xi.
	\endeq
	Taking $\xi=(0, 0, 0)$, we have $\pi_{S_1} P_{\xi}(V)=\pi_{S_1} (V)$. Hence $\text{dim}(\pi_{S_1}(V))\le 1$. This further implies $\text{dim}(\pi_{S_2}(V))\ge 1$. We will consider two cases.
	\begin{description}
		\item[Case 1. $\text{dim}(\pi_{S_2}(V))=2.$] In this case we have $\pi_{S_2}(V)=S_2$. By a direct calculation,
		$$
		\text{dim} (\pi_{S_1} P_{\xi}(V))\geq \text{dim} (\pi_{S_1} P_{\xi} (\pi_{S_2}(V)))=
		$$
		$$
		\text{dim} (\pi_{S_1} P_{\xi}(S_2))=\text{rank}
		\begin{bmatrix}
		\partial_r Q_1(\xi) & \partial_s Q_1(\xi) & \partial_t Q_1(\xi) \\
		\partial_r Q_2(\xi) & \partial_s Q_2(\xi) & \partial_t Q_2(\xi)
		\end{bmatrix}
		$$
		which equals 2 almost surely in $\xi$, by \eqref{lab2}. This is a contradiction to \eqref{guo8.9}.
		
		\item[Case 2. $\text{dim}(\pi_{S_2}(V))=1$.] In this case $\text{dim}(\pi_{S_1}(V))=1$. Also, $\pi_{S_1}P_{\xi}(V)$ is a subspace of $\pi_{S_1}(P_{\xi}\pi_{S_1}(V))+\pi_{S_1}P_{\xi}(S_2)$ of co-dimension at most one.Observe that $\pi_{S_1}(P_{\xi}\pi_{S_1}(V))=\pi_{S_1}(V)$. Suppose that $\pi_{S_1}(V)$ is spanned by the non-zero vector $(u, v, w)\in \R^3$. Then the dimension of the space $\pi_{S_1}(V)+\pi_{S_1}P_{\xi}(S_2)$ is given by
		\beq\label{guo8.12}
		\text{rank}\begin{bmatrix}
			\partial_r Q_1(\xi) & \partial_s Q_1(\xi) & \partial_t Q_1(\xi) \\
			\partial_r Q_2(\xi) & \partial_s Q_2(\xi) & \partial_t Q_2(\xi) \\
			u & v & w
		\end{bmatrix}
		\endeq
		which, by the assumption of Theorem \ref{1803thm1.1}, equals three almost surely in $\xi$. Hence $\pi_{S_1}P_{\xi}(V)$ is at least 2 almost surely in $\xi$. This is again a contradiction to \eqref{guo8.9}.
	\end{description}
	We have finished the proof of the case $\text{dim}(V)=2$.\\
	
	In the end we consider the case $\text{dim}(V)=4$. We will again argue by contradiction. Suppose that
	\beq\label{guo8.13}
	\text{dim}(\pi_{S_1}P_{\xi}(V))\le 2 \text{ for every } \xi.
	\endeq
	Then we obtain $\pi_{S_1}(V)\le 2$ as before. Therefore $\text{dim}( \pi_{S_2}(V))=2$. Hence $\text{dim}(\pi_{S_1}(V))=2$ and $V=\pi_{S_1}(V)\oplus S_2$. Take a non-zero vector $(u, v, w)\in \pi_{S_1}(V)$. Then the dimension of $\pi_{S_1}P_\xi(V)$ is at least equal to the rank from \eqref{guo8.12}, which, by the assumption of Theorem \ref{1803thm1.1}, is three almost surely in $\xi$. This leads to a contradiction to \eqref{guo8.13}. Thus we have finished the proof of the case $\text{dim}(V)=4$.
\end{proof}

\section{Other related manifolds}
\label{s7}

Let $D_{\mc{M}}(N,p)$ be the $l^pL^p$ decoupling constant associated with a $d$-dimensional  manifold $\mc{M}$ in $\R^n$
$$\mc{M}=\{(t_1,\ldots,t_d,\phi_1(t_1,\ldots,t_d),\ldots,\phi_{n-d}(t_1,\ldots,t_d)),\;t_i\in[0,1]\}.$$
The functions $\phi_i$ need not necessarily be quadratic, just continuous.  We claim the following universal lower bound
\beq
\label{ww33}
D_{\mc{M}}(N,p)\gtrsim_{\mc{M}} \max\{N^{\frac{d}{2}(\frac12-\frac1p)},N^{\frac{d}{2}-\frac{n}{p}}\},\;\;p\ge 2.
\endeq
Let us see why this holds true. Theorem 2.2 in \cite{BD2} extends easily to our generality here. It implies that
for  each $p\ge 2$ and each $a_{i_1,\ldots,i_d}\in\C$, $0\le i_1,\ldots,i_d\le N$ we have
\begin{equation*}
	\begin{split}
		\Big(\frac1{N^{2n}}\int_{[0,N^2]^n} & |\sum_{i_1=0}^N\ldots\sum_{i_d=0}^Na_{i_1,\ldots,i_d}  e(x_1\frac{i_1}N+\ldots+x_d\frac{i_d}{N}+\\
		& \sum_{j=d+1}^nx_j\phi_j(\frac{i_1}N,\ldots, \frac{i_d}N))|^{p}dx_1\ldots dx_n\Big)^{\frac1p} \lesssim  D_{\mc{M}}(N^2,p)\|a_{i_1,\ldots,i_d}\|_{l^p}.
	\end{split}
\end{equation*}
Let us now specialize to the case $a_{i_1,\ldots,i_d}\equiv 1$. We get
\begin{equation}\label{fek20}
	\begin{split}
		\Big(\frac1{N^{2n}}\int_{[0,N^2]^n} & |\sum_{i_1=0}^N\ldots\sum_{i_d=0}^Ne(x_1\frac{i_1}N+\ldots+x_d\frac{i_d}{N}+\\
		& \sum_{j=d+1}^nx_j\phi_j(\frac{i_1}N,\ldots, \frac{i_d}N))|^{p}dx_1\ldots dx_n\Big)^{\frac1p}\lesssim  D_{\mc{M}}(N^2,p)N^{\frac{d}{p}}.
	\end{split} 
\end{equation}
We present two lower bounds for \eqref{fek20}. The first is obtained by rewriting \eqref{fek20} (using periodicity) as follows
\begin{multline*}
	\Big(\frac1{N^{2n-d}}\int_{[0,N]^d\times [0,N^2]^{n-d}}|\sum_{i_1=1}^N\ldots\sum_{i_d=1}^Ne(x_1\frac{i_1}N+\ldots+x_d\frac{i_d}{N}+\\
	\sum_{j=d+1}^nx_j\phi_j(\frac{i_1}N,\ldots, \frac{i_d}N))|^{p}dx_1\ldots dx_n\Big)^{\frac1p}
\end{multline*}
and by restricting $|x_1|,\ldots,|x_n|\lesssim_{\mc{M}} 1$. This restriction will almost align the phases of exponentials and will produce the lower bound
$$N^{d}N^{-\frac{2n-d}{p}}.$$
Using H\"older provides the following second lower bound for \eqref{fek20}
\begin{multline*}
	\Big(\frac1{N^{2n}}\int_{[0,N^2]^n}|\sum_{i_1=1}^N\ldots\sum_{i_d=1}^Ne(x_1\frac{i_1}N+\ldots+x_d\frac{i_d}{N}+\\
	\sum_{j=d+1}^nx_j\phi_j(\frac{i_1}N,\ldots, \frac{i_d}N))|^{2}dx_1\ldots dx_n\Big)^{\frac12}.
\end{multline*}
This term is of order $N^{\frac{d}{2}}$, which can be seen by invoking $L^2$ quasi-orthogonality.  Now \eqref{ww33} follows by combining these two lower bounds.

These considerations suggest the following question.

\begin{question}
	Is it true that for each $n>d\ge 1$ there exists a $d$-dimensional  manifold in $\R^n$ whose decoupling constant satisfies
	\beq
	\label{ww34}
	D_{\mc{M}}(N,p)\lesssim_{\epsilon} N^{\epsilon}\max\{N^{\frac{d}{2}(\frac12-\frac1p)},N^{\frac{d}{2}-\frac{n}{p}}\}
	\endeq
	for all $p\ge 2$?
\end{question}
Note that this upper bound is trivially true for $p=2,\infty$. By invoking interpolation as in \cite{BD1}, \eqref{ww34} is equivalent with the inequality
\beq
\label{ww35}
D_{\mc{M}}(N,p)\lesssim_{\epsilon} N^{\frac{d}{2}(\frac12-\frac1p)+\epsilon}
\endeq
for $2\le p\le p_c=\frac{4n}{d}-2$. The largest $p_c$ for which \eqref{ww35} holds for $2\le p\le p_c$ is the  so-called critical exponent for the $l^pL^p$ decoupling for $\mc{M}$. The question we asked is whether there is a manifold for which $p_c=\frac{4n}{d}-2$. It seems likely that the answer is ``yes" at least when $d>\frac{n}3$.

By combining all previous results on decouplings, we have a positive answer in the case $(d,n)=(n-1,n)$ (hypersurfaces, see \cite{BD1}) for all $n\ge 2$. Other known cases are $(d,n)=(2,4)$ (see \cite{BD5}), $(2,5)$ (see \cite{BD2}) and $(2,9)$ (see \cite{BDGu}). And of course, we can now add $(3,5)$. An interesting case for which the above question is open is $d=1$, for all $n\ge 3$. The end of the paper \cite{BD10} contains a discussion with the state of the art for $d=1$. In particular, it proves that \eqref{ww34} holds in some range $2\le p\le p_n$, for some $p_n<4n-2$.
\bigskip

\section{Appendix}
In this Appendix, we will show that the assumption of Theorem \ref{1803thm1.1}, that is, for each nonzero vector $(u,v,w)\in \R^3$
\beq\label{guo10.1}
\det \begin{bmatrix}\frac{\partial Q_1}{\partial r} & \frac{\partial Q_1}{\partial s}& \frac{\partial Q_1}{\partial t}\\\\\frac{\partial Q_2}{\partial r} & \frac{\partial Q_2}{\partial s}& \frac{\partial Q_2}{\partial t}\\\\u & v& w \end{bmatrix}
\endeq
is a nonzero polynomial, is equivalent to Lemma \ref{2201lemma3.1} being true. This is the same as saying, that if one intends to prove the decoupling inequalities \eqref{1803e1.7} via the Bourgain--Demeter  multi-linear approach, then the assumption of Theorem \ref{1803thm1.1} is indeed necessary. Hence it would be reasonable to believe that for two quadratic functions $Q_1, Q_2$ not satisfying this assumption, the desired bound \eqref{1803e1.7} would fail. \\

More specifically, if we define
\beq
\mathcal{Z}:=\{(u, v, w)\in \R^3 \Big| \eqref{guo10.1} \text{ is constantly zero}\},
\endeq
then we will prove the following result.
\begin{lem}
	If $\text{dim}(\mathcal{Z})\ge 1$, then we can find a subspace $V\subset \R ^5$ of dimension $2$ (or $4$ resp.) such that
	$$\{(r, s, t): dim(\pi_{r, s, t}(V))\le 1  (\text{ resp. } 2)\} $$ is the whole space $\R^3$ .
\end{lem}
\begin{proof}
	We introduce some notation. Let $Q_1 (r,s,t)= \frac{1}{2} (A_1 r^2 +A_2 s^2 +A_3 t^2) +A_4 rs +A_5 rt +A_6 st$ and $Q_2 (r,s,t)= \frac{1}{2} (B_1 r^2 +B_2 s^2 +B_3 t^2) +B_4 rs +B_5 rt +B_6 st$ be two homogeneous polynomials of degree two. Let
	\beq
	d_{ij} := A_i B_j -A_j B_i \text{ for } 1\leq i,j \leq 6.
	\endeq
	We will split the proof into three cases, according to the dimension of $\mathcal{Z}$.\\
	
	First, assume $\dim(\mathcal{Z})=3$. Then we obtain that all the two by two minors of the matrix
	\beq\label{guo10.4}
	\begin{bmatrix}\frac{\partial Q_1}{\partial r} & \frac{\partial Q_1}{\partial s}& \frac{\partial Q_1}{\partial t}\\\\\frac{\partial Q_2}{\partial r} & \frac{\partial Q_2}{\partial s}& \frac{\partial Q_2}{\partial t} \end{bmatrix}
	\endeq
	have constantly vanishing determinants. By a direct calculation, this further implies $d_{ij}=0$ for all $1\leq i,j \leq 6$. Hence we obtain that $a Q_1 + b Q_2\equiv 0$ for some non-zero $(a, b)\in \R^2$. In the end, we take
	\beq
	V= \text{span} \{(1,0,0,0,0),(0,0,0,a,b)\}
	\endeq
	and it is easy to see that $\dim{\pi_{r, s, t}(V)}\le 1$ for every $(r,s,t) \in \R ^3$. This finishes the proof of the case $\text{dim}(\mathcal{Z})=3.$\\
	
	%\eqref{appk}, $d_{ij}=0$ for all $1\leq i,j \leq 6$. Hence $Q_1, Q_2$ are linearly dependent, say $a Q_1 + b Q_2\equiv 0$ for some nonzero $(a,b)$. We could take $V= \text{span} \{(1,0,0,0,0),(0,0,0,a,b)\}$ and it is easy to see that $\dim{\pi_{r, s, t}(V)}\le 1$ for any $(r,s,t) \in \R ^3$.\\\
	
	Next, we assume $\dim(\mathcal{Z})=2$. Let $\mathcal{Z}=\text{span}\{(u_1,v_1,w_1), (u_2,v_2,w_2)\}$. Let
	\beq
	V:=\text{span} \{(u_1,v_1,w_1,0,0),(u_2,v_2,w_2,0,0),(0,0,0,1,0),(0,0,0,0,1)\}.
	\endeq
	We claim that
	\beq
	\dim{\pi_{r, s, t}(V)}\le 2 \text{ for all } (r,s,t)\in \R^3.
	\endeq
	By the rank-nullity theorem, this is equivalent to the fact that
	\beq\label{guo10.8}
	\text{rank} \begin{bmatrix}\frac{\partial Q_1}{\partial r} & \frac{\partial Q_1}{\partial s}& \frac{\partial Q_1}{\partial t}\\\\\frac{\partial Q_2}{\partial r} & \frac{\partial Q_2}{\partial s}& \frac{\partial Q_2}{\partial t}\\\\u_1 & v_1& w_1\\\\ u_2& v_2& w_2 \end{bmatrix}\leq 2
	\endeq
	everywhere in $\R^3$. To prove this claim, we will first do a change of variables to make future computations simpler. To be precise, for a nonsingular linear transformation $M$ from $\R^3$ to $\R^3$, let $\widetilde{Q_i} := Q_i \circ M$. Correspondingly, we define $\widetilde{\mathcal{Z}}$.
	By Remark \ref{rmnf897-594090-ei=rfo=[=1[} it is easy to see that $\widetilde{\mathcal{Z}}=M^{T}\mathcal{Z}$ and the claim that
	%$$\text{rank} \begin{bmatrix}\frac{\partial Q_1}{\partial r} & \frac{\partial Q_1}{\partial s}& \frac{\partial Q_1}{\partial t}\\\\\frac{\partial Q_2}{\partial r} & \frac{\partial Q_2}{\partial s}& \frac{\partial Q_2}{\partial t}\\\\u_1 & v_1& w_1\\\\ u_2& v_2& w_2 \end{bmatrix}\leq 2 $$
	\eqref{guo10.8} holds everywhere is equivalent to the fact that
	\beq\label{guo10.9}
	\text{rank} \begin{bmatrix}\frac{\partial \widetilde{Q_1}}{\partial r} & \frac{\partial \widetilde{Q_1}}{\partial s}& \frac{\partial \widetilde{Q_1}}{\partial t}\\\\\frac{\partial \widetilde{Q_2}}{\partial r} & \frac{\partial \widetilde{Q_2}}{\partial s}& \frac{\partial \widetilde{Q_2}}{\partial t}\\\\ \widetilde{u_1} & \widetilde{v_1}& \widetilde{w_1}\\\\ \widetilde{u_2}& \widetilde{v_2}& \widetilde{w_2} \end{bmatrix}\leq 2
	\endeq
	everywhere, where $(\widetilde{u_i} , \widetilde{v_i}, \widetilde{w_i})= (u_i,v_i,w_i) M$. We now choose $M$ so that
	\beq
	(\widetilde{u_1} , \widetilde{v_1}, \widetilde{w_1})=(1, 0, 0) \text{ and } (\widetilde{u_2} , \widetilde{v_2}, \widetilde{w_2})=(0, 1, 0).
	\endeq
	This condition tells us that
	\beq
	\det \begin{bmatrix} \frac{\partial \widetilde{Q_1}}{\partial s} & \frac{\partial \widetilde{Q_1}}{\partial t}\\\\\frac{\partial \widetilde{Q_2}}{\partial s} & \frac{\partial \widetilde{Q_2}}{\partial t}
	\end{bmatrix} = \det \begin{bmatrix} \frac{\partial \widetilde{Q_1}}{\partial r} & \frac{\partial \widetilde{Q_1}}{\partial t}\\\\\frac{\partial \widetilde{Q_2}}{\partial r} & \frac{\partial \widetilde{Q_2}}{\partial t}
	\end{bmatrix} \equiv 0 \text{ and } \det \begin{bmatrix} \frac{\partial \widetilde{Q_1}}{\partial r} & \frac{\partial \widetilde{Q_1}}{\partial s}\\\\\frac{\partial \widetilde{Q_2}}{\partial r} & \frac{\partial \widetilde{Q_2}}{\partial s}
	\end{bmatrix}\not\equiv 0,
	\endeq
	as otherwise $(0, 0, 1)\in \widetilde{\mathcal{Z}}$, which is a contradiction to $\text{dim}(\widetilde{\mathcal{Z}})=2$. We further conclude that
	\beq
	\frac{\partial \widetilde{Q_1}}{\partial t}=\frac{\partial \widetilde{Q_2}}{\partial t}\equiv 0,
	\endeq
	which implies the desired estimate \eqref{guo10.9}.
	To see this, we argue by contradiction. If not, then for almost every $(r,s,t)\in \R^3$, we could find $a(r,s,t), b(r,s,t) \in \R$ such that
	$$\begin{bmatrix} \frac{\partial \widetilde{Q_1}}{\partial s} \\\\\frac{\partial \widetilde{Q_2}}{\partial s}
	\end{bmatrix} = a(r,s,t) \begin{bmatrix} \frac{\partial \widetilde{Q_1}}{\partial t} \\\\\frac{\partial \widetilde{Q_2}}{\partial t}
	\end{bmatrix}  \;\;\text{   and   }\;\; \begin{bmatrix} \frac{\partial \widetilde{Q_1}}{\partial r} \\\\\frac{\partial \widetilde{Q_2}}{\partial r}
	\end{bmatrix} = b(r,s,t) \begin{bmatrix} \frac{\partial \widetilde{Q_1}}{\partial t} \\\\\frac{\partial \widetilde{Q_2}}{\partial t}
	\end{bmatrix} .$$
	Hence
	$$\det \begin{bmatrix} \frac{\partial \widetilde{Q_1}}{\partial r} & \frac{\partial \widetilde{Q_1}}{\partial s}\\\\\frac{\partial \widetilde{Q_2}}{\partial r} & \frac{\partial \widetilde{Q_2}}{\partial s}
	\end{bmatrix}=0$$ almost everywhere. This contradicts the fact that the zero set of a nonzero polynomial has Lebesgue measure zero.\\\
	
	Finally, we look at the case $\dim(\mathcal{Z})=1$. Let $\mathcal{Z}= \text{span} \{(u,v,w)\}$. We claim that there exists a nonzero vector $(x,y)\in \R^2$ such that
	\beq\label{guo10.13}
	\text{rank} \begin{bmatrix} u& v& w\\\  x\frac{\partial Q_1}{\partial r}+y\frac{\partial Q_2}{\partial r} & x\frac{\partial Q_1}{\partial s}+ y\frac{\partial Q_2}{\partial s} & x\frac{\partial Q_1}{\partial t} +y\frac{\partial Q_2}{\partial t} \end{bmatrix} \leq 1
	\endeq
	everywhere. This, if true, combined with the rank-nullity theorem, will imply that $\dim{\pi_{r, s, t}(V)}\leq 1$ for all $(r,s,t)\in \R^3$ with
	\beq
	V:=\text{span} \{(u,v,w,0,0),(0,0,0,x,y)\}.
	\endeq
	%Assume that this is true for the moment. Let $$V:=\text{span} \{(u,v,w,0,0),(0,0,0,x,y)\}.$$ It is easy to see that $\dim{\pi_{r, s, t}(V)}\leq 1$ for all $(r,s,t)\in \R^3.$
	
	Thus, what is left is to prove that \eqref{guo10.13} holds true everywhere in $\R^3$. We will make a change of variables similar to the one in the previous case. We also adopt the notation from there. Let $M$ be a $3\times 3$ nonsingular matrix such that $(1,0,0)= (u,v,w) M$. Then $\widetilde{\mathcal{Z}}=\text{span} \{(1,0,0)\} $ and \eqref{guo10.13}
	%$$\text{rank} \begin{bmatrix} u& v& w\\\  x\frac{\partial Q_1}{\partial r}+y\frac{\partial Q_2}{\partial r} & x\frac{\partial Q_1}{\partial s}+ y\frac{\partial Q_2}{\partial s} & x\frac{\partial Q_1}{\partial t} +y\frac{\partial Q_2}{\partial t} \end{bmatrix} \leq 1 $$
	everywhere is equivalent with
	\beq\label{guo10.15}
	\text{rank} \begin{bmatrix} 1& 0& 0\\\  x\frac{\partial \widetilde{Q_1}}{\partial r}+y\frac{\partial \widetilde{Q_2}}{\partial r} & x\frac{\partial \widetilde{Q_1}}{\partial s}+ y\frac{\partial \widetilde{Q_2}}{\partial s} & x\frac{\partial \widetilde{Q_1}}{\partial t} +y\frac{\partial \widetilde{Q_2}}{\partial t} \end{bmatrix} \leq 1
	\endeq
	everywhere. From now on, we will drop the tilde notation and assume that our original linear space $\mathcal{Z}$ is spanned by $(1,0,0)$. Hence \eqref{guo10.13} is equivalent with finding a non-zero $(x,y)\in \R^2$ such that
	\beq\label{appkk}
	(x, y)\begin{bmatrix} \frac{\partial Q_1}{\partial s} & \frac{\partial Q_1}{\partial t}\\\\\frac{\partial Q_2}{\partial s} & \frac{\partial Q_2}{\partial t}
	\end{bmatrix}\equiv 0.
	\endeq
	Recall that $\mathcal{Z}$ is spanned by the vector $(1, 0, 0)$. This implies
	\beq\label{guo10.17}
	\det \begin{bmatrix} \frac{\partial Q_1}{\partial s} & \frac{\partial Q_1}{\partial t}\\\\\frac{\partial Q_2}{\partial s} & \frac{\partial Q_2}{\partial t}
	\end{bmatrix}
	=\det \begin{bmatrix}
		A_2 s+ A_4 r+ A_6 t & A_3 t+A_5 r+A_6 s\\
		B_2 s+ B_4 r+ B_6 t & B_3 t+B_5 r+B_6 s
	\end{bmatrix}
	\equiv 0.
	\endeq
	Hence \eqref{appkk} is equivalent to saying that the two-dimensional vector $(\frac{\partial Q_1}{\partial t}, \frac{\partial Q_2}{\partial t})$ does not change directions, which, by a direct calculation, is further equivalent to
	\beq\label{guo10.18}
	d_{36}=d_{35}=d_{56}=0.
	\endeq
	To prove \eqref{guo10.18}, we make a second change of variables. The goal of this change of variables is to make $A_6=B_6=0$. This will automatically imply $d_{36}=d_{56}=0$. Moreover, we would like to keep the space $\mathcal{Z}$ unchanged, that is, we want $(1, 0, 0)$ to be invariant under this change of variables. This can be realized by choosing a linear transformation $M$ of the form
	\beq\label{guo10.19}
	M=\begin{bmatrix}
		1 & 0 & 0\\
		0 & m_{22} & m_{23}\\
		0 & m_{32} & m_{33}\\
	\end{bmatrix}
	\endeq
	with $M'=\begin{bmatrix}m_{22} & m_{23}\\m_{32}&m_{33} \end{bmatrix}$ being some non-singular $2\times 2$ matrix. After this linear transformation, we will obtain two new quadratic functions $Q_i'=Q_i\circ M$ for $i\in \{1, 2\}$. Again for the sake of simplicity, we will keep using the original notation $Q_i$ instead of $Q_i'$.
	
	It is straightforward to see that there exists a  linear transformation of the form \eqref{guo10.19} that sends at least one of $A_6$ and $B_6$ to zero. Indeed, what this $M$ does is to keep the $r$ variable unchanged and to diagonalize the quadratic form in $s,t$ variables.  Let us show that the other one will be zero simultaneously. Recall that the space $\mathcal{Z}$ is still spanned by the vector $(1, 0, 0)$. Hence the relation \eqref{guo10.17} still holds. Letting $r=0$, \eqref{guo10.17} further implies that
	\beq
	\text{rank}\begin{bmatrix}
		A_2 & A_3 & A_6\\
		B_2 & B_3 & B_6
	\end{bmatrix}
	\le 1.
	\endeq
	Hence the two quadratic forms
	$$\frac{1}{2} (A_2 s^2+A_3 t^2) +A_6 st,     \frac{1}{2}(B_2 s^2 + B_3 t^2) + B_6 st$$
	are linearly dependent. So the matrix \eqref{guo10.19} can be chosen such that $A_6=B_6=0$ at the same time.\\
	
	To prove \eqref{guo10.18}, what remains is to prove $d_{35}=0$. We argue by contradiction. Assume $d_{35}\neq 0.$ We look at the assumption \eqref{guo10.17}. By setting $A_6=B_6=0$, we obtain
	\beq
	d_{23}=d_{25}=0 \text{ and } d_{43}=d_{45}=0.
	\endeq
	These, combined with the assumption that $d_{35}\neq 0$, further imply that $A_2=B_2=0$ and $A_4=B_4=0$. Together with $A_6=B_6=0$, we conclude that $\frac{\partial Q_1}{\partial s}=\frac{\partial Q_2}{\partial s}\equiv 0$. Hence the vector $(0, 0, 1)$ also belongs to $\mathcal{Z}$, which means $\text{dim}(\mathcal{Z})\ge 2$. This contradicts the assumption that $\text{dim}(\mathcal{Z})=1$. This finishes the proof of the case $\text{dim}(\mathcal{Z})=1$, thus the proof of the whole lemma.
\end{proof}

%%%%%%%%%%%%%%%%%%%%%%%%%%%%%%%%%
% Bibliography
%%%%%%%%%%%%%%%%%%%%%%%%%%%%%%%%%

\noindent Department of Mathematics, Indiana University, 831 East 3rd St., Bloomington IN 47405\\
\emph{Email address}: demeterc@indiana.edu\\

\noindent Department of Mathematics, Indiana University, 831 East 3rd St., Bloomington IN 47405\\
\emph{Email address}: shaomingguo@math.wisc.edu\\

\noindent Department of Mathematics, Indiana University, 831 East 3rd St., Bloomington IN 47405\\
\emph{Email address}: shif@indiana.edu

\end{document}